\documentclass{amsart}
\usepackage{amssymb}
\usepackage{amsfonts}
\usepackage{amsmath}
\usepackage{psfig}
\usepackage{psfrag}
\usepackage{graphicx}

\newtheorem{thm}{Theorem}[section]
\newtheorem{pro}[thm]{Proposition}
\newtheorem{cor}[thm]{Corollary}
\newtheorem{lem}[thm]{Lemma}
\newtheorem{defi}[thm]{Definition}
\newtheorem{rmk}[thm]{Remark}
\newtheorem{ex}[thm]{Example}

\renewcommand{\theequation}{\thesection.\arabic{equation}}

\begin{document}

\title{The natural measure of a symbolic dynamical system}


\author{Wen-Guei Hu$^{*}$}
\address{Department of Applied Mathematics, National Chiao Tung University, Hsinchu 300, Taiwan}
\email{wghu@mail.nctu.edu.tw}\thanks{$^{*}$The author would like to thank the National Science Council, R.O.C. and
the ST Yau Center for partially supporting this research.}

\author{Song-Sun Lin$^{\dagger}$}
\address{Department of Applied Mathematics, National Chiao Tung University, Hsinchu 300, Taiwan}
\email{sslin@math.nctu.edu.tw} \thanks{$^{\dagger}$The author would like to thank the National Science Council, R.O.C. (Contract No. NSC 101-2115-M-009-007) and
the ST Yau Center for partially supporting this research.}

\begin{abstract}
This study investigates the natural or intrinsic measure of a symbolic dynamical system $\Sigma$. The measure $\mu([i_{1},i_{2},\cdots,i_{n}])$ of a pattern \\ $[i_{1},i_{2},\cdots,i_{n}]$ in $\Sigma$ is an asymptotic ratio of $[i_{1},i_{2},\cdots,i_{n}]$, which arises in all patterns of length $n$ within very long patterns, such that in a typical long pattern, the pattern $[i_{1},i_{2},\cdots,i_{n}]$ appears with frequency $\mu([i_{1},i_{2},\cdots,i_{n}])$. When $\Sigma=\Sigma(A)$ is a shift of finite type and $A$ is an irreducible $N\times N$ non-negative matrix, the measure $\mu$ is the Parry measure. $\mu$ is ergodic with maximum entropy. The result holds for sofic shift $\mathcal{G}=(G,\mathcal{L})$, which is irreducible. The result can be extended to $\Sigma(A)$, where $A$ is a countably infinite matrix that is irreducible, aperiodic and positive recurrent. By using the Krieger cover, the natural measure of a general shift space is studied in the way of a countably infinite state of sofic shift, including context free shift. The Perron-Frobenius Theorem for non-negative matrices plays an essential role in this study.
\end{abstract}

\maketitle

\section{Introduction}
This paper investigates the natural or intrinsic measure of a symbolic dynamical system. The measure is an average of allowable words or admissible patterns, which are made of the system.

Let $\mathcal{S}_{N}=\{1,2,\cdots,N\}$ be the set of symbols with $N\geq 2$, and denote by
$\Sigma(N)$ the set of all two-sided sequences with symbols in $\mathcal{S}_{N}$:

\begin{equation}\label{eqn:1.1}
\Sigma(N)\equiv \mathcal{S}_{N}^{\mathbf{Z}^{1}}=\left\{(x_{n})_{n=-\infty}^{\infty} \hspace{0.1cm}\mid\hspace{0.1cm} x_{n}\in \mathcal{S}_{N}  \right\}.
\end{equation}
Let $\sigma$ be the shift map that is defined by
\begin{equation*}
(\sigma x)_{n}=x_{n+1}.
\end{equation*}
Then, $(\Sigma(N),\sigma)$ is the two-sided full shift space. $\Sigma(N)$ is compact when it is equipped with product topology.
A shift space $\Sigma\subset \Sigma(N)$ is closed and invariant under $\sigma$. $\Sigma$ is also equipped with the topology that is reduced from $\Sigma(N)$.

A cylinder in $\Sigma$ is given by
\begin{equation}\label{eqn:1.2}
C([\alpha;i_{1},i_{2},\cdots, i_{n}])=\{x\in\Sigma  \hspace{0.1cm}\mid\hspace{0.1cm} x_{\alpha}=i_{1}, x_{\alpha+1}=i_{2},\cdots, x_{\alpha+n-1}=i_{n}\}
\end{equation}
where $i_{k}\in\mathcal{S}_{N}$, $1\leq k\leq n$ and $\alpha\in\mathbb{Z}^{1}$ is the initial site. $[i_{1},i_{2},\cdots, i_{n}]$ is called an admissible pattern if $C([\alpha;i_{1},i_{2},\cdots, i_{n}])\neq\emptyset$ for some $\alpha\in\mathbb{Z}^{1}$. Denote by $\mathcal{B}_{n}(\Sigma)$ the set of all admissible patterns or allowable words of length $n$:

\begin{equation}\label{eqn:1.3}
\mathcal{B}_{n}(\Sigma)=\left\{[i_{1},i_{2},\cdots, i_{n}] \hspace{0.1cm}\mid\hspace{0.1cm} C([\alpha;i_{1},i_{2},\cdots, i_{n}])\neq\emptyset \text{ for some }\alpha\in\mathbb{Z}^{1}\right\}.
\end{equation}
Let $\mathcal{B}(\Sigma)=\underset{n=1}{\overset{\infty}{\bigcup}}\mathcal{B}_{n}(\Sigma)$. $\Sigma$ is called irreducible if for any ordered pair $U_{1}$ and $U_{2}\in\mathcal{B}(\Sigma)$, there is a $W\in\mathcal{B}(\Sigma)$ such that $U_{1}W U_{2}\in \mathcal{B}(\Sigma)$.

Let $\mathfrak{B}(\Sigma)$ be the $\sigma$-algebra that is generated by the collection $\mathcal{C}$ of all cylinders:

\begin{equation}\label{eqn:1.4}
\mathcal{C}(\Sigma)=\left\{C([\alpha;i_{1},i_{2},\cdots, i_{n}])  \hspace{0.1cm}\mid\hspace{0.1cm} i_{k}\in\mathcal{S}_{N}, 1\leq k\leq n \text{ and }\alpha\in\mathbb{Z}^{1}  \right\}.
\end{equation}
Then, $(\Sigma,\mathfrak{B}(\Sigma))$ is a measurable space.

If $\sigma$ is ergodic with respect to $(\Sigma, \mathfrak{B}(\Sigma), m)$, where $m$ is a probability measure, then the Birkhoff ergodic theorem implies that for any $E\in \mathfrak{B}(\Sigma)$, and a.e. $x\in\Sigma$,

\begin{equation}\label{eqn:1.4-1}
\frac{1}{n}\left|\left\{i\in\{0,1,\cdots,n-1\}\hspace{0.1cm}\mid \hspace{0.1cm}\sigma^{i}(x)\in E\right\}\right| \rightarrow m(E),
\end{equation}
, where $|S|=\sharp\{S\}$ is the cardinality of finite set $S$.
Therefore, a natural or intrinsic ergodic measure for $(\Sigma,\mathfrak{B}(\Sigma))$ that can better describe the dynamical system $\Sigma$ would be of particular value.

In this work, an intrinsic probability measure $\mu$ on $(\Sigma,\mathfrak{B}(\Sigma))$ is obtained by taking the asymptotic ratios of the admissible patterns $C\left([i_{1},i_{2},\cdots, i_{n}]\right)$ to  $\mathcal{B}_{n}(\Sigma)$  by the following process.

The simplest asymptotic ratio is defined as follows:
Let the finite cylinder $C_{k,l}([i_{1},i_{2},\cdots, i_{n}])$ of length $n+k+l$ be

\begin{equation}\label{eqn:1.5}
\begin{array}{rl}
 & C_{k,l}([i_{1},i_{2},\cdots, i_{n}]) \\
 & \\
  = & \left\{ x=(x_{-k+1},\cdots,x_{-1},x_{0},x_{1},\cdots,x_{n+l} )\in  \mathcal{B} _{n+k+l}(\Sigma)\hspace{0.1cm}\mid\hspace{0.1cm}x_{j}=i_{j}, 1\leq j\leq n \right\}
\end{array}
\end{equation}
and define

\begin{equation}\label{eqn:1.6}
\begin{array}{rl}
 & \mu([i_{1},i_{2},\cdots, i_{n}]) \\
 & \\
=  & \underset{k,l\rightarrow\infty}{\lim} \left|C_{k,l}([i_{1},i_{2},\cdots, i_{n}]) \right|/ \left|\mathcal{B}_{n+k+l}(\Sigma) \right|
\end{array}
\end{equation}
whenever the limit exists.
Equation (\ref{eqn:1.6}) can be interpreted as the ratio of $n$-patterns $[i_{1},i_{2},\cdots, i_{n}]$ within a window of width $n$ for long admissible patterns.

First, the shift space $\Sigma$ that has a limit that is given by (\ref{eqn:1.6}) is identified and the properties of $\mu$ are then investigated. For example, $\mu$ is ergodic, has maximal entropy and so on. When $\mu$ is a measure that is defined as in (\ref{eqn:1.6}), $\mu$ is called a natural or an intrinsic measure, since it is defined by the intrinsic properties of $\Sigma$, depending only on its own admissible patterns.

When $\mu$ is ergodic, then (\ref{eqn:1.4-1}) and (\ref{eqn:1.6}) imply that the time average of the orbit of almost all $x$ tends the space average (ratio of admissible patterns):

\begin{equation}\label{eqn:1.6-1}
\begin{array}{rl}
 & \frac{1}{m}\left|\left\{i\in\{0,1,\cdots,m-1\}\hspace{0.1cm}\mid \hspace{0.1cm}\sigma^{i}(x)\in C([i_{1},i_{2},\cdots, i_{n}])\right\} \right| \\ & \\\sim  & C_{k,l}([i_{1},i_{2},\cdots, i_{n}])/\mathcal{B}_{n+k+l}(\Sigma)
\end{array}
\end{equation}
for large $m$, $k$ and $l$. Hence, for a dynamical system and a typical orbit, a pattern $\left[i_{1},i_{2},\cdots,i_{n}\right]$ appears with frequency $\mu\left(\left[i_{1},i_{2},\cdots,i_{n}\right]\right)$.

The simplest shift space that ensures the existence of natural measure that is given by (\ref{eqn:1.6}) is shift of finite type $\Sigma(A)$, where $A=[a_{i,j}]_{N\times N}$ is an irreducible and aperiodic $0$-$1$ matrix.
In this case, the measure $\mu(A)$ turns out to be the well-known Parry measure \cite{95}, which is given by

\begin{equation}\label{eqn:1.6-2}
\mu([i_{1},i_{2},\cdots, i_{n}])=\frac{u_{i_{1}}v_{i_{n}}}{\lambda^{n-1}}a_{i_{1},i_{2}}a_{i_{2},i_{3}}\cdots a_{i_{n-1},i_{n}},
\end{equation}
where $\lambda$ is the Perron value of $A$, and $V=(v_{1},v_{2},\cdots,v_{N})^{t}$ and $U=(u_{1},u_{2},\cdots, u_{N})$ are the right and left eigenvectors with respect to $\lambda$, normalized by $UV=1$.

 If $A$ is only irreducible but not aperiodic, then the limit of (\ref{eqn:1.6}) does not exist; see Example 3.4. When $A$ is periodic with period $p$, (\ref{eqn:1.6}) can be replaced by

\begin{equation}\label{eqn:1.7}
 \mu\left(\left[i_{1},i_{2},\cdots, i_{n} \right]\right)=\underset{k,l\rightarrow\infty}{\lim} \hspace{0.15cm} \frac{1}{p} \underset{j=0}{\overset{p-1}{\sum}} \left|C_{k-j,l+j}([i_{1},i_{2},\cdots, i_{n}]) \right|/ \left|\mathcal{B}_{n+k+l}(\Sigma) \right|,
\end{equation}
which is the average of $p$ successive patterns.

Again, the measure $\mu$ is the Parry measure and (\ref{eqn:1.6-2}) holds. Therefore, the following results are obtained.

\begin{thm}
\label{theorem:1.1}
If $A=[a_{i,j}]_{N\times N}$ is irreducible and aperiodic, then the natural measure that is defined by (\ref{eqn:1.6}) exists and equals the Parry measure. If $A$ is irreducible and periodic with period $p\geq2$, then the natural measure defined by  (\ref{eqn:1.7}) exists and equals the Parry measure. Furthermore, the natural measure is the only measure that is ergodic and has maximal entropy.
\end{thm}
%
%
%
%
%

Now, consider the sofic shift, which is a factor of shift of finite type. Usually, sofic shift is represented by a labeled graph $\mathcal{G}$, which is a pair $\mathcal{G}=(G,\mathcal{L})$ where the labeling $\mathcal{L}$ assigns to each edge of graph $G$ a label from finite alphabets $\mathcal{S}=\left\{s_{1},s_{2},\cdots, s_{K}\right\}$.
Sofic shift is called irreducible (or aperiodic) if $G$ is irreducible (or aperiodic). $\mathcal{G}$ is right-resolving if the edges carry different labels at each vertex.

The adjacency matrix $\mathbb{A}_{s}=[a_{s;i,j}]$ of the alphabet $s$ is defined by $a_{s;i,j}=1$ when the alphabet $s$
 appears on the edge from vertex $i$ to $j$. Otherwise, $a_{s;i,j}=0$. The (total) adjacency matrix is given by $\mathbb{A}=\underset{i=1}{\overset{K}{\sum}}\mathbb{A}_{s_{i}}$.

Denote by $\mathcal{B}_{n}(X_{\mathcal{G}})$ the set of all admissible patterns with length $n$ in $X_{\mathcal{G}}$ and let

\begin{equation*}
\begin{array}{rl}
 & C_{k,l}([s_{i_{1}},s_{i_{2}},\cdots,s_{i_{n}}]) \\
 & \\
  = & \left\{ x=(x_{-k+1},\cdots,x_{-1},x_{0},x_{1},\cdots,x_{n+l} )\in  \mathcal{B}_{n+k+l}(X_{\mathcal{G}})\hspace{0.1cm}\mid\hspace{0.1cm}x_{j}=s_{i_{j}}, 1\leq j\leq n \right\}.
\end{array}
\end{equation*}
Then, the natural measure $\mu=\mu_{\mathcal{G}}$ can be defined in a manner similar to the definitions (\ref{eqn:1.6}) and (\ref{eqn:1.7}), as follows. When $\mathbb{A}$ is irreducible and aperiodic,

\begin{equation}\label{eqn:1.7-1}
\mu\left(\left[s_{i_{1}},s_{i_{2}},\cdots,s_{i_{n}}\right]\right)=\underset{k,l\rightarrow\infty}{\lim} \left|C_{k,l}([s_{i_{1}},s_{i_{2}},\cdots,s_{i_{n}}]) \right|/ \left|\mathcal{B}_{n+k+l}(X_{\mathcal{G}}) \right|;
\end{equation}
when $\mathbb{A}$ is irreducible with period $p\geq 2$,

\begin{equation}\label{eqn:1.7-2}
 \mu\left(\left[s_{i_{1}},s_{i_{2}},\cdots,s_{i_{n}}\right]\right)=\underset{k,l\rightarrow\infty}{\lim} \hspace{0.15cm} \frac{1}{p} \underset{j=0}{\overset{p-1}{\sum}} \left|C_{k-j,l+j}([s_{i_{1}},s_{i_{2}},\cdots,s_{i_{n}}]) \right|/ \left|\mathcal{B}_{n+k+l}(X_{\mathcal{G}}) \right|.
\end{equation}

The natural measure of sofic shift can be obtained as follows.

\begin{thm}
\label{theorem:1.2}
Assume sofic shift $\mathcal{G}=(G,\mathcal{L})$ is irreducible and right-resolving; then, the natural measure $\mu_{\mathcal{G}}$ exists and is given by

\begin{equation}\label{eqn:1.8}
\mu_{\mathcal{G}}\left(\left[s_{i_{1}},s_{i_{2}},\cdots,s_{i_{n}}\right]\right)=\frac{1}{\lambda^{n}}\underset{i}{\sum}\underset{j}{\sum} u_{i} (\mathbb{A}_{s_{i_{1}}}\mathbb{A}_{s_{i_{2}}}\cdots\mathbb{A}_{s_{i_{n}}})_{i,j}v_{j},
\end{equation}
where $\lambda$ is the Perron value of $\mathbb{A}$, and $V=(v_{1},v_{2},\cdots,v_{n})^{t}$ and $U=(u_{1},u_{2},\cdots,u_{n})$ with $UV=1$ are the associated right and left eigenvectors. The natural measure is the only measure with ergodicity and maximal entropy. Furthermore, the natural measure $\mu_{\mathcal{G}}$ is the hidden Markov measure of the underlying shift of finite type of graph $G$.
\end{thm}

When $\mathcal{G}=(G,\mathcal{L})$ is a shift of finite type and $A$ is the adjacency matrix of $G$, then (\ref{eqn:1.8}) is reduced to (\ref{eqn:1.6-2}).

The previous results can be generalized further to shift space $\Sigma_{T}$, where $T$ is a countably infinite zero-one matrix. Indeed, let $\mathcal{I}$ of countably infinite symbols be the index set of $T$. Denote by $\mathcal{B}_{n;i,j}(\Sigma(T))$ the admissible patterns with initial state $i$ and terminal state $j$:

\begin{equation*}
\mathcal{B}_{n;i,j}(\Sigma(T))=\left\{[i_{1},i_{2},\cdots, i_{n}]\in \mathcal{B}_{n}(\Sigma(T)) \hspace{0.1cm}\mid\hspace{0.1cm} i_{1}=i \text{ and } i_{n}=j \right\}
\end{equation*}
and

\begin{equation*}
\begin{array}{rl}
 & C_{k,l;i,j}([i_{1},i_{2},\cdots, i_{n}]) \\
 & \\
  = & \left\{ x=(x_{-k+1},\cdots,x_{-1},x_{0},x_{1},\cdots,x_{n+l} )\in  \mathcal{B} _{n+k+l}(\Sigma(T))\hspace{0.1cm}\mid\hspace{0.1cm}x_{-k+1}=i, x_{n+l}=j, x_{q}=i_{q}, 1\leq q\leq n \right\}.
\end{array}
\end{equation*}
Define

\begin{equation}\label{eqn:1.8-0}
\begin{array}{rl}
 & \mu_{i,j}([i_{1},i_{2},\cdots, i_{n}]) \\
 & \\
=  & \underset{k,l\rightarrow\infty}{\lim} \left|C_{k,l;i,j}([i_{1},i_{2},\cdots, i_{n}]) \right|/ \left|\mathcal{B}_{n+k+l;i,j}(\Sigma(T)) \right|.
\end{array}
\end{equation}

Recall the Perron-Frobenius Theorem for countable state $T$. Suppose that $T$ is irreducible, aperiodic and recurrent. Then, the Perron value

\begin{equation}\label{eqn:1.8-1}
\lambda=\underset{n\rightarrow\infty}{\lim}\sqrt[n]{(T^{n})_{i,j}}
\end{equation}
for any index $i$ and $j$. Assume $\lambda$ is finite. Let $\mathbf{l}$ and $\mathbf{r}$ be the left and right eigenvectors of $\lambda$. $T$ is called positive recurrent if

\begin{equation}\label{eqn:1.8-2}
\mathbf{l}\cdot\mathbf{r}<\infty .
\end{equation}
The following result is thus obtained.

\begin{thm}
\label{theorem:1.2}
Let $T$ be irreducible, aperiodic and positive recurrent. Then, the natural measure exists and equals

\begin{equation}\label{eqn:1.8-3}
\begin{array}{rl}
\mu([i_{1},i_{2},\cdots ,i_{n}])= & \mu_{i,j}([i_{1},i_{2},\cdots ,i_{n}])\\
& \\
= & \frac{l_{i_{1}}r_{i_{n}}}{\lambda^{n-1}}T_{i_{1},i_{2}}T_{i_{2},i_{3}}\cdots T_{i_{n-1},i_{n}},
\end{array}
\end{equation}
which is independent of $i$ and $j$, where the left and right eigenvectors $\mathbf{l}=\left(l_{j}\right)$ and $\mathbf{r}=\left(r_{j}\right)$ are normalized such that $\mathbf{l}\cdot\mathbf{r}=1$. Furthermore, $\mu$ has maximal entropy of $\log \lambda$.
\end{thm}
Equation (\ref{eqn:1.8-3}) is the infinite-dimensional version of the Parry measure. A similar result holds for sofic shifts with countable states; see Theroem 5.5 for details.

Finally, the Krieger cover can be used to reduce some shift spaces to a sofic shift with countably infinite states.
For such shift spaces, under the Krieger cover with countable states, the adjacency matrix $T$ is irreducible, aperiodic and positive recurrent, has the natural measure that is ergodic and has maximal entropy. The well-known context free shift belongs to this case and is studied in detail here.

The finiteness of $\mathbf{l}\cdot\mathbf{r}<\infty$ in (\ref{eqn:1.8-2}) plays an essential role in establishing Theorem 1.3.
When (\ref{eqn:1.8-2}) fails, the cases in which $\mathbb{T}$ is null recurrent or transient, or in which the shift space has uncountable Krieger cover, need separate investigation. Indeed, in the case of random walk on integers, no measure can be obtained from (\ref{eqn:1.8-0}); see Example 5.4.

On the other hand, the study of the pattern generation problems of higher dimensional symbolic dynamical systems have some progress in recent years; see for examples \cite{0-1,0-2,0-3,0-4,101,102-0,0-5,1} and references therein. This work can be extended to higher-dimensional cases.

The rest of this paper is organised as follows. Section 2 reviews some useful results from measure theory and symbolic dynamics. Section 3 studies shifts of finite type and proves Theorem 1.1. Section 4 concerns sofic shifts and proves Theorem 1.2. Section 5 reviews shifts of finite type and sofic shifts with countable states. The existence of a natural measure is proven when the adjacency matrix is irreducible and positive recurrent. Section 6 introduces Krieger cover for general shift spaces and the existence of the natural measure is proven when the associated adjacency matrix is countable, irreducible and positive recurrent.

\section{Preliminaries}
\setcounter{equation}{0}

This section recalls some useful notation and results from measure theory and symbolic dynamics.

\subsection{Measure theory}
\label{sec:2-1}
Let $X$ be a set. A $\sigma$-algebra of a subset of $X$ is a collection $\mathfrak{B}$ of subsets of $X$ such that (i) $X\in\mathfrak{B}$; (ii) if $B\in\mathfrak{B}$ then $X\setminus B\in\mathfrak{B}$; (iii) $\underset{n=1}{\overset{\infty}{\cup}}B_{n}\in \mathfrak{B}$ whenever $B_{n}\in\mathfrak{B}$ for all $n\geq 1$. Given $\sigma$-algebra $\mathfrak{B}$, $(X,\mathfrak{B})$ is called a measurable space.

A function $m$ is a finite measure on $(X,\mathfrak{B})$ if $m$ is a non-negative function on $\mathfrak{B}$ such that $m: \mathfrak{B}\rightarrow[0,\infty)$ satisfies (i) $m(\emptyset)=0$; (ii) $m(\underset{n=1}{\overset{\infty}{\cup}}B_{n})=\underset{n=1}{\overset{\infty}{\sum}}m(B_{n})$ where $B_{i}\cap B_{j}=\emptyset$ for $i\neq j$. $(X, \mathfrak{B},m)$ is a probability space if $m(X)=1$.
In this study, $(X, \mathfrak{B},m)$ is always assumed to be a probability space.

Suppose $(X_{i}, \mathfrak{B}_{i},m_{i})$, $i=1,2$, are probability spaces. A transformation $T:X_{1}\rightarrow X_{2}$ is called measurable if $T^{-1}(B_{2})\in \mathfrak{B}_{1}$ for any $B_{2}\in\mathfrak{B}_{2}$. $T$ is called measure-preserving if

\begin{equation*}
m_{1}(T^{-1}(B_{2}))=m_{2}(B_{2})
\end{equation*}
for all $B_{2}\in\mathfrak{B}_{2}$. If $T: (X,\mathfrak{B},m)\rightarrow (X,\mathfrak{B},m)$ is measure-preserving, then $T$ is called ergodic when $T^{-1}(B)=B$ implies $m(B)=0$ or $m(B)=1$. Furthermore, $T$ is called strong mixing if $\underset{n\rightarrow\infty}{\lim}m\left(T^{-n}(A)\cap B\right)=m(A)m(B)$ for all $A$ and $B$.

Recall the well-known Birkhoff Ergodic Theorem.

\begin{thm}[Birkhoff Ergodic Theorem]
\label{theorem:2.1}
Suppose $T: (X,\mathfrak{B},m)\rightarrow (X,\mathfrak{B},m)$ is a measure-preserving map and $f\in L^{1}(m)$. Then

\begin{equation*}
\frac{1}{n} \underset{i=0}{\overset{n-1}{\sum}}f(T^{i}(x))\rightarrow f^{*} \hspace{0.2cm} \text{a.e.}
\end{equation*}
for some $f^{*}\in L^{1}(m)$, and $f^{*}\circ T=f^{*}$ a.e. and $\int f^{*} dm = \int f dm$.

Furthermore if $T$ is ergodic then $f^{*}$ is a constant a.e. with
\begin{equation*}
f^{*}=\int f dm \hspace{0.2cm} \text{a.e.,}
\end{equation*}
i.e.

\begin{equation}\label{eqn:2.1}
\underset{n\rightarrow\infty}{\lim}\hspace{0.1cm}\underset{i=0}{\overset{n-1}{\sum}}\hspace{0.1cm}f(T^{i}(x))= \int f dm.
\end{equation}
\end{thm}

Let $(X,\mathfrak{B},m)$ be a probability space and $T: X\rightarrow X$ be ergodic. Then, for any $E\in\mathfrak{B}$,

\begin{equation*}
\frac{1}{n}\left|\left\{ i\in\{0,1,\ldots,n-1\} \hspace{0.1cm}\mid\hspace{0.1cm} T^{i}(x)\in E\}\right\}\right|\rightarrow m(E) \hspace{0.2cm} \text{a.e.}
\end{equation*}
by the ergodic theorem, the orbit of almost every point in $X$ enters $E$ with asymptotic relative frequency $m(E)$.

Given a finite sub-algebra $\mathcal{A}$ of $\mathfrak{B}$, $\mathcal{A}$ forms a  partition $\xi(\mathcal{A})=\{A_1,A_2,\cdots ,A_k\}$ of $(X,\mathfrak{B},m)$. The entropy of $\mathcal{A}$ (or $\xi(\mathcal{A})$) is defined by

\begin{equation*}
H_{m}(\mathcal{A})=H_{m}(\xi(\mathcal{A}))=-\underset{i=1}{\overset{k}{\sum}}m(A_i)\log m(A_i).
\end{equation*}
Suppose $T: X\rightarrow X$ is a measure-preserving transformation of $(X,\mathfrak{B}, m)$. The entropy of $T$ with respect to $\mathcal{A}$ is defined by

\begin{equation*}
h_{m}(T,\mathcal{A})=\underset{n\rightarrow\infty}{\lim}\frac{1}{n}H_{m}\left(\underset{i=0}{\overset{n-1}{\bigvee}} T^{-i}\mathcal{A}\right).
\end{equation*}
The existence of the limit can be demonstrated by using the subadditive of $H_{}\left(\underset{i=0}{\overset{n-1}{\bigvee}} T^{-i}\mathcal{A}\right)$ \cite{2}.
Then, the entropy $h_{m}(T)$ of $T$ is given by $h_{m}(T)=\sup h_{m}(T,\mathcal{A})$, where the supremum is taken over all finite sub-algebras $\mathcal{A}$ of $\mathfrak{B}$.

Now, the variational principle is recalled as follows.

\begin{thm}
\label{theorem:2.3}
Let $T: X\rightarrow X$ be a continuous map of a compact metric space $X$. Then $h(T)=\sup \left\{h_{m}(T) \hspace{0.1cm}\mid \hspace{0.1cm}  m\in M(X,T)\right\}$,
where $h(T)$ is the topological entropy of $T$.
\end{thm}

Based on the variational principle, a measure $m$ of $M(X,T)$ is called a measure of maximal entropy if $h_{m}(T)=h(T)$.

\subsection{Shift space}
\label{sec:2-2}
This section introduces some useful notation and results concerning shift space.

Let $\mathcal{S}_{N}=\{1,2,\cdots,N\}$ be the set of $N$ symbols (or alphabets) with $N\geq 2$. The two-sided sequence space is

\begin{equation*}
\Sigma(N)\equiv \mathcal{S}_{N}^{\mathbb{Z}^{1}}=\left\{(x_{n})_{n=-\infty}^{\infty} \hspace{0.1cm}\mid\hspace{0.1cm} x_{n}\in \mathcal{S}_{N} \text{ for all }n\in\mathbb{Z}^{1}\right\}.
\end{equation*}
%

The shift $\sigma$ on $\Sigma(N)$ is $(\sigma x)_{n}=x_{n+1}$ for $x\in\Sigma(N)$. Therefore, $(\Sigma(N),\sigma)$ is the full shift space. As usual, $\Sigma(N)$ is equipped with product topology. Then, $\Sigma(N)$ is a compact topological space by Tychonoff's theorem.

A subshift $\Sigma$ of $\Sigma(N)$ is a closed and shift-invariant subset of $\Sigma(N)$. An important subshift is a shift of finite type, which is defined as follows.

Let $A$ be an $N\times N$ square $0$-$1$ matrix. The subshift $\Sigma(A)$ is defined by

\begin{equation*}
\Sigma(A)\equiv\left\{x\in \Sigma(N)\hspace{0.1cm}\mid\hspace{0.1cm} A_{x_{n},x_{n+1}}=1 \text{ for all }n\in\mathbb{Z}^{1} \right\}.
\end{equation*}
$\Sigma(A)$ is called a shift of finite type.

Given shift space $\Sigma$, let $\Sigma_{n}$ be the set of the admissible patterns with length $n$. Then, the topological entropy $h(\Sigma)$ of $\Sigma$ is defined by

\begin{equation}\label{eqn:2.1-1}
h(\Sigma)= \underset{n\rightarrow\infty}{\lim}\frac{1}{n}\log |\Sigma_{n}|.
\end{equation}
The subadditive of $\log |\Sigma_{n}|$ implies that the limit always exists.

Perron-Frobenius Theory for a non-negative matrix $A$ is required. $A=[a_{i,j}]$ is called irreducible if, for each $i$ and $j$, $n=n(i,j)$ exists such that $(A^{n})_{i,j}>0$. $A$ is called primitive or irreducible and aperiodic if $n$ is independent of $i$ and $j$, and so there exists $n$ such that $A^{n}$ is positive.

\begin{thm}[Perron-Frobenius Theorem]
\label{theorem:2.2}
Let $A=[a_{i,j}]$ be a non-negative $k\times k$ matrix.

\begin{itemize}
\item[(i)] There is a non-negative eigenvalue $\lambda$ such that no eigenvalue of $A$ has absolute value greater than $\lambda$.

\item[(ii)] Corresponding to the eigenvalue $\lambda$, there is a non-negative left eigenvector $u=(u_{1},u_{2},\cdots, u_{k})$ and a non-negative right eigenvector $v=(v_{1},v_{2},\cdots, v_{k})^{t}$.

\item[(iii)] If $A$ is irreducible, then $\lambda$ is a simple eigenvalue and the corresponding eigenvectors are strictly positive (i.e. $u_{i}>0$, $v_{i}>0$ for all $i$).

\item[(iv)] If $A$ is irreducible and aperiodic, then $\lambda$ is strictly larger than the absolute values of other eigenvalues.
\end{itemize}

\end{thm}
The non-negative eigenvalue obtained in the Perron-Frobenius theorem is called the Perron Value.

The following result follows from the Perron-Frobenius theorem, and it is very useful in studying the properties of the natural measure; see Theorem 0.17 of Walters \cite{2}.

\begin{pro}
\label{proposition:2.2-1}
Assume $A$  is irreducible and aperiodic, and $\lambda$ is the Perron value of $A$ with normalized right eigenvector $V=(v_{1},v_{2}, \cdots, v_{N})^{t}$ and left eigenvector $U=(u_{1},u_{2}, \cdots, u_{N})$ with $UV=1$. Then

\begin{equation}\label{eqn:2.1-4}
\underset{k\rightarrow\infty}{\lim}\frac{A^{k}}{\lambda^{k}}=[v_{i}u_{j}].
\end{equation}
If $A$ is irreducible, then
\begin{equation}\label{eqn:2.1-2}
\underset{K\rightarrow \infty}{\lim} \frac{1}{K}\underset{k=0}{\overset{K-1}{\sum}}\left(\frac{A^{k}}{\lambda^{k}}\right)=[v_{i}u_{j}].
\end{equation}

\end{pro}

A non-negative $N\times N$ matrix $P=[p_{i,j}]$ is called stochastic if $\underset{j=1}{\overset{N}{\sum}}p_{i,j}=1$. Let $p=(p_{1},p_{2},\cdots, p_{N})$ be the invariant probability vector for $P$ such that $pP=p$ and $\underset{i=1}{\overset{N}{\sum}}p_{i}=1$. Then, $(p,P)$ induces a probability space $(\Sigma(P), \mathfrak{B}(P), m(P))$ by

\begin{equation}\label{eqn:2.1-3}
m([i_{1},i_{2},\cdots, i_{n}])=p_{i_{1}}P_{i_{1},i_{2}}\cdots P_{i_{n-1},i_{n}}.
\end{equation}
$(p,P)$ is called a Markov shift \cite{2}. Moreover, the shift map $\sigma$ on $(p,P)$ is ergodic.

When $A=[a_{i,j}]$ is an irreducible nonnegative matrix. By the Perron-Frobenius theorem, $A$ induces the stochastic matrix $P=[P_{i,j}]$ with

\begin{equation}\label{eqn:2.2}
P_{i,j}=\frac{a_{i,j}v_{j}}{\lambda v_{i}}.
\end{equation}
Equations (\ref{eqn:2.1-3}) and (\ref{eqn:2.2}) imply

\begin{equation}\label{eqn:2.6}
m([i_{1},i_{2},\cdots, i_{n}])=\frac{u_{i_{1}}v_{i_{n}}}{\lambda^{n-1}}a_{i_{1},i_{2}}a_{i_{2},i_{3}}\cdots a_{i_{n-1},i_{n}}.
\end{equation}

Therefore, for $(\Sigma(A), \mathfrak{B}(A), m(A))$, $m(A)$ that is defined through (\ref{eqn:2.6}) is called the Parry measure, which is the most important probability measure of shift space $\Sigma(A)$ as follows.

\begin{thm}
\label{theorem:2.3}
If $A$ is irreducible, then Parry measure is the only measure with ergodicity and maximal entropy. Furthermore, if $A$ is irreducible and aperiodic, then the Parry measure is strong mixing.
\end{thm}
See Walters \cite{2}.

\section{Shift of finite type}
\setcounter{equation}{0}
This section investigates shifts of finite type and shows that when $A$ is irreducible, the natural measure exists and equals the Parry measure. The case with irreducible and aperiodic $A$ is considered first.

\begin{thm}
\label{theorem:3.1}
If $A=[a_{i,j}]_{N\times N}$ is irreducible and aperiodic. Then the natural measure that is defined by (\ref{eqn:1.6}) exists and equals the Parry measure.
\end{thm}

\begin{proof}
Let $\lambda_{1}$ be the Perron value of $A$. The Perron-Frobenius theorem implies

\begin{equation}\label{eqn:3.1}
\lambda_{1}>|\lambda_{2}|\geq |\lambda_{3}|\geq\cdots \geq |\lambda_{N}|,
\end{equation}
where $\lambda_{j}$ are eigenvalues of $A$. Let $U=(u_{1},u_{2},\cdots,u_{N})$ and $V=(v_{1},v_{2},\cdots,v_{N})^{t}$ be the left and right eigenvectors of $A$ with respect to $\lambda_{1}$:

\begin{equation}\label{eqn:3.2}
\begin{array}{ccc}
AV=\lambda_{1}V  & \text{and} & UA=\lambda_{1}U.
\end{array}
\end{equation}
$U$ and $V$ are normalized such that $\underset{i=1}{\overset{N}{\sum}}u_{i}v_{i}=1$. Using (\ref{eqn:3.1}), the Jordan form of $A$ is

\begin{equation}\label{eqn:3.3}
H^{-1}AH=J=\left[
\begin{array}{cccc}
\lambda_{1} & 0 & \cdots & 0 \\
 0 & b_{2,2} & \cdots& b_{2,N}\\
 \vdots & \vdots&\ddots& \vdots\\
 0 & b_{N,2}&\cdots &b_{N,N}
\end{array}
\right]_{N\times N},
\end{equation}
where
\begin{equation}\label{eqn:3.4}
\begin{array}{ccc}
H=\left[
\begin{array}{cccc}
v_{1} & c_{1,2} & \cdots&  c_{1,N}\\
v_{2} & c_{2,2} & \cdots & c_{2,N} \\
\vdots &\vdots & \ddots& \vdots\\
v_{N} &  c_{N,2}& \cdots&  c_{N,N}\\
\end{array}
\right]
&
\text{and}
&
H^{-1}=\left[
\begin{array}{cccc}
u_{1} & u_{2}&\cdots & u_{N}\\
 d_{2,1} & d_{2,2} &\cdots & d_{2,N} \\
 d_{3,1}& d_{3,2} & \cdots&  d_{3,N}\\
 \vdots & \vdots & \ddots & \vdots \\
  d_{N,1}&  d_{N,2}&\cdots &  d_{N,N}\\
\end{array}
\right].
\end{array}
\end{equation}
Then, (\ref{eqn:3.1}), (\ref{eqn:3.3}) and (\ref{eqn:3.4}) imply the following results:

\begin{equation}\label{eqn:3.5}
|A^{m}|=\lambda_{1}^{m}\left(\underset{q=1}{\overset{N}{\sum}}v_{q}\right)\left(\underset{q=1}{\overset{N}{\sum}}u_{q}\right)+f_{m}(\lambda_{2},\cdots, \lambda_{N}),
\end{equation}

\begin{equation}\label{eqn:3.6}
\underset{q=1}{\overset{N}{\sum}}(A^{k})_{q,i}=\lambda_{1}^{k}\left(\underset{q=1}{\overset{N}{\sum}}v_{q}\right)u_{i}+g_{k}(\lambda_{2},\cdots, \lambda_{N})
\end{equation}
and
\begin{equation}\label{eqn:3.7}
\underset{q=1}{\overset{N}{\sum}}(A^{l})_{j,q}=\lambda_{1}^{l}\left(\underset{q=1}{\overset{N}{\sum}}u_{q}\right)v_{j}+h_{l}(\lambda_{2},\cdots, \lambda_{N})
\end{equation}
with

\begin{equation}\label{eqn:3.8}
\begin{array}{ccc}
f_{m}/\lambda_{1}^{m}\rightarrow 0, & g_{k}/\lambda_{1}^{k}\rightarrow 0, & h_{l}/\lambda_{1}^{l}\rightarrow 0
\end{array}
\end{equation}
as $m$, $k$, $l$ tend to $\infty$, respectively. Here, $|M|=\underset{i=1}{\overset{N}{\sum}}\underset{j=1}{\overset{N}{\sum}}M_{i,j}$ is the summation of all entries for matrix $M=[M_{i,j}]_{N\times N}$.

For any $n\geq 1$, it is easy to see that
\begin{equation}\label{eqn:3.9}
\begin{array}{rl}
& \left|C_{k,l}([i_{1},i_{2},\cdots, i_{n}])\right|\\
& \\
= & \underset{q=1}{\overset{N}{\sum}} (A^{k})_{q,i_{1}} a_{i_{1},i_{2}}\cdots a_{i_{n-1},i_{n}} \underset{q'=1}{\overset{N}{\sum}} (A^{l})_{i_{n},q'}.
\end{array}
\end{equation}
Hence, (\ref{eqn:3.5})$\sim$(\ref{eqn:3.9}) imply

\begin{equation}\label{eqn:3.10}
\begin{array}{rl}
\mu([i_{1},i_{2},\cdots, i_{n}])= & \underset{k,l\rightarrow\infty}{\lim}\left|C_{k,l}([i_{1},i_{2},\cdots, i_{n}])\right|/|A^{k+l+n}| \\
& \\
= & \frac{u_{i_{1}}v_{i_{n}}}{\lambda_{1}^{n-1}}a_{i_{1},i_{2}}a_{i_{2},i_{3}}\cdots a_{i_{n-1},i_{n}}.
\end{array}
\end{equation}
Therefore, from (\ref{eqn:2.6}) and (\ref{eqn:3.10}), $\mu$ is the Parry measure.

\end{proof}

Now, consider $A$, an $N\times N$ irreducible matrix with period $p\geq 2$. $A$ has $p$ eigenvalues whose absolute value equals $\lambda_{1}$:

\begin{equation*}
\begin{array}{cccccc}
\lambda_{1}, &  \lambda_{2}=\lambda_{1}e^{i 2\pi/p }, &  \cdots, & \lambda_{j}=\lambda_{1}e^{i 2\pi (j-1)/p }, &  \cdots, & \lambda_{p}=\lambda_{1}e^{i 2\pi (p-1)/p };
\end{array}
\end{equation*}
the absolute values of the other eigenvalues $\lambda_{k}$, $p+1\leq k\leq N$, are strictly smaller than $\lambda_{1}$.

By permutation transformation, $A$ can be expressed in the following form:

\begin{equation}\label{eqn:3.14}
 A=\left[
 \begin{array}{ccccc}
 0 & A_{1,2} & 0 & \cdots & 0 \\
  0 & 0 & A_{2,3} & \cdots & 0 \\
   \vdots & \vdots & \ddots & \ddots & \vdots \\
    0 & 0 & 0 & \cdots & A _{p-1,p} \\
     A_{p,1} & 0& 0 & \cdots & 0 \\
 \end{array}
 \right]
\end{equation}
and
\begin{equation}\label{eqn:3.15}
 A^{p}=\left[
 \begin{array}{ccccc}
 B_{1} & 0  & 0 & \cdots & 0 \\
  0 &  B_{2} & 0 & \cdots & 0 \\
   0 & 0 &  B_{3}& \cdots & 0 \\
 \vdots & \vdots & \ddots & \ddots & \vdots \\
    0 & 0& 0 & \cdots &  B_{p} \\
 \end{array}
 \right],
\end{equation}
where $B_{i}=A_{i,i+1}A_{i+1,i+2}\cdots A_{p-1,p} A_{p,1} \cdots A_{i-1,i}$ is irreducible and aperiodic, $1\leq i\leq p$.

For $1\leq i\leq p$, let $n_{i}$ be the size of $B_{i}$ and

\begin{equation}\label{eqn:3.16}
\mathcal{D}_{i}=\left\{1\leq j\leq N \hspace{0.1cm}\mid\hspace{0.1cm} \left(\underset{k=1}{\overset{i-1}{\sum}}n_{k}\right)+1\leq j\leq \underset{k=1}{\overset{i}{\sum}}n_{k} \right\}.
\end{equation}

From the Perron-Frobenius Theorem and the form (\ref{eqn:3.14}), the following theorem can be easily proven. The proof is omitted here.

\begin{thm}
\label{theorem:3.2}
Let $A$ be a irreducible matrix with period $p$ of the form (\ref{eqn:3.14}), $p\geq 2$, and $V=(v_{1},v_{2},\cdots,v_{N})^{t}$ be the positive right eigenvector of the Perron value $\lambda_{1}$. Then, the right eigenvector $V^{(j)}=\left(v^{(j)}_{1},v^{(j)}_{2},\cdots,v^{(j)}_{N}\right)^{t}$, $2\leq j\leq p$, corresponding to $\lambda_{j}=\lambda_{1}e^{i 2\pi (j-1)/p }$ is given by
\begin{equation}\label{eqn:3.16}
v^{(j)}_{k}= e^{i 2\pi (j-1)(b-1) /p }v_{k} \hspace{0.3cm}\text{ if }\hspace{0.3cm} k\in \mathcal{D}_{b}.
\end{equation}
On the other hand, let  $U=(u_{1},u_{2},\cdots,u_{N})$ be the positive left eigenvector of $\lambda_{1}$.
Then, the left eigenvector $U^{(j)}=\left(u^{(j)}_{1},u^{(j)}_{2},\cdots,u^{(j)}_{N}\right)$, $2\leq j\leq p$, corresponding to $\lambda_{j}$ is given by
\begin{equation}\label{eqn:3.17}
u^{(j)}_{k}= e^{i 2\pi (j-1)(-b+1) /p }u_{k} \hspace{0.3cm}\text{ if }\hspace{0.3cm} k\in \mathcal{D}_{b}.
\end{equation}

\end{thm}
Therefore, by (\ref{eqn:1.7}), the natural measure of matrix $A$, which is irreducible and periodic with period $p\geq 2$, can be shown to equal the Parry measure.

From now on, for clarity of notation, denote by $f(n)\sim g(n)$ if
\begin{equation*}
\underset{n\rightarrow\infty}{\lim}\frac{f(n)}{g(n)}=1.
\end{equation*}

\begin{thm}
\label{theorem:3.1}
If $A=[a_{i,j}]_{ N\times N}$ is irreducible, then the natural measure that is defined by (\ref{eqn:1.7}) exists and equals the Parry measure.
\end{thm}
%
%

\begin{proof}
Without loss of generality, assume that $A$ is of the form (\ref{eqn:3.14}).
It is easy to verify that if $n+k+l-1\equiv r$ (mod $p$), by Theorem 3.2, then
\begin{equation*}
\begin{array}{rl}
\left|\mathcal{B}_{n+k+l}(\Sigma) \right|=\left|A^{n+k+l-1} \right|\sim &  \underset{q=1}{\overset{p}{\sum}} \left(\underset{i=1}{\overset{n}{\sum}}v^{(q)}_{i}\right)\left(\underset{i=1}{\overset{n}{\sum}}u^{(q)}_{i}\right)\lambda_{q}^{n+k+l-1} \\
& \\
= & p\left[ \underset{q=1}{\overset{p}{\sum}} \left( \underset{i\in\mathcal{D}_{q}}{\sum}v_{i} \right)  \left( \underset{j\in\mathcal{D}_{q'}}{\sum}u_{j}\right)  \right]\lambda_{1}^{n+l+k-1},
\end{array}
\end{equation*}
where
\begin{equation*}
q'=\left\{
\begin{array}{ll}
q+r & \text{ if }1\leq q+r\leq p,\\
& \\
q+r-p & \text{ if }q+r> p.
\end{array}
\right.
\end{equation*}

Clearly, $[i_{1},i_{2},\cdots,i_{n}]$ is an admissible pattern if and only if $a_{i_{1},i_{2}}\cdots a_{i_{n-1},i_{n}}=1$. From Theorem 3.2, it can be proven that for any admissible pattern $[i_{1},i_{2},\cdots,i_{n}]$,

\begin{equation*}
\begin{array}{rl}
 & \underset{j=0}{\overset{p-1}{\sum}} \left|C_{k-j,l+j}([i_{1},i_{2},\cdots, i_{n}])\right| \\ & \\= &   \underset{j=0}{\overset{p-1}{\sum}} \left(\underset{q=1}{\overset{N}{\sum}} (A^{k-j})_{q,i_{1}}  \right)\left(\underset{q=1}{\overset{N}{\sum}} (A^{l+j})_{i_{n},q}  \right) \\
 & \\
 \sim & p^{2} \left( \underset{q=1}{\overset{p}{\sum}} \left( \underset{i\in\mathcal{D}_{q}}{\sum}v_{i} \right)  \left( \underset{j\in\mathcal{D}_{q'}}{\sum}u_{j}\right) \right) u_{i_{1}}v_{i_{n}}\lambda_{1}^{k+l}.
\end{array}
\end{equation*}
For brevity, the details are omitted.

Therefore,
\begin{equation*}
\underset{k,l\rightarrow\infty}{\lim} \hspace{0.15cm} \frac{1}{p}\underset{j=0}{\overset{p-1}{\sum}}\frac{\left|C_{k-j,l+j}([i_{1},i_{2},\cdots, i_{n}])\right|}{\left|\mathcal{B}_{n+k+l}(\Sigma) \right|}=\frac{u_{i_{1}}v_{i_{n}}}{\lambda_{1}^{n-1}}a_{i_{1},i_{2}}a_{i_{2},i_{3}}\cdots a_{i_{n-1},i_{n}}.
\end{equation*}
By (\ref{eqn:1.7}), the result follows.
\end{proof}
The following example illustrates Theorems 3.2 and 3.3.

\begin{ex}
Consider $H=\left[\begin{array}{cccc}
0 & 1 & 1 & 1 \\
1 & 0 & 0 & 0 \\
1 & 0 & 0 & 0 \\
1 & 0 & 0 & 0 \\
\end{array}
\right]$. Clearly, $H$ is irreducible with period $2$.

The eigenvalues are $\lambda_{1}=\sqrt{3}$, $\lambda_{2}=-\sqrt{3}$, $\lambda_{3}=\lambda_{4}=0$. The right eigenvector corresponding to $\lambda_{1}$ is $V=\frac{1}{\sqrt{6}}(\sqrt{3},1,1,1)^{t}$, and the right eigenvector corresponding to $\lambda_{2}$ is $V^{(2)}=\frac{1}{\sqrt{6}}(\sqrt{3},-1,-1,-1)^{t}$. On the other hand, the left eigenvector corresponding to $\lambda_{1}$ is $U=V^{t}$ and the left eigenvector corresponding to $\lambda_{2}$ is $(V^{(2)})^{t}$. Notably, $UV=1$.

From Theorems 3.2 and 3.3, the natural measure equals the Parry measure with probability vector $p=(\frac{1}{2},\frac{1}{6},\frac{1}{6},\frac{1}{6})$ and stochastic matrix $P=\left[\begin{array}{cccc}
0 & \frac{1}{3} &  \frac{1}{3}  &  \frac{1}{3}  \\
1 & 0 & 0 & 0 \\
1 & 0 & 0 & 0 \\
1 & 0 & 0 & 0 \\
\end{array}
\right]$.

It can be verified that
\begin{equation*}
\begin{array}{ccc}
H^{2n}=  3^{n-1} \left[
\begin{array}{cccc}
3 & 0 & 0 & 0 \\
0 & 1 & 1 & 1 \\
0 & 1 & 1 & 1 \\
0 & 1 & 1 & 1
\end{array}
\right]
& \text{and} &
H^{2n+1}=  3^{n} \left[
\begin{array}{cccc}
0 & 1 & 1 & 1\\
1 & 0 & 0 & 0 \\
1 & 0 & 0 & 0  \\
1 & 0 & 0 & 0
\end{array}
\right].
\end{array}
\end{equation*}
Then, the limit (\ref{eqn:1.6}) with even length
\begin{equation*}
m_{e}([i_{1},i_{2},\cdots, i_{n}])
\equiv\underset{n+k+l:\text{ even}}{\underset{k,l\rightarrow\infty}{\lim}} \left|C_{k,l}([i_{1},i_{2},\cdots, i_{n}]) \right|/ \left|\mathcal{B}_{n+k+l}(\Sigma) \right|
\end{equation*}
equals the Parry measure.
However, the limit (\ref{eqn:1.6}) with odd length
\begin{equation*}
\underset{n+k+l:\text{ odd}}{\underset{k,l\rightarrow\infty}{\lim}} \left|C_{k,l}([i_{1},i_{2},\cdots, i_{n}]) \right|/ \left|\mathcal{B}_{n+k+l}(\Sigma) \right|
\end{equation*}
oscillates and does not exist. Averaging the even and odd lengths according to (\ref{eqn:1.7}) reveals that the limit exists and equals the Parry measure.
\end{ex}
%
%
%
%
%

Now, suppose $A$ is a reducible $N\times N$ matrix. In this case, there exists a permutation matrix such that $A$ can be expressed as a block upper triangular matrix:

\begin{equation}\label{eqn:3.16-8}
A=\left[
\begin{array}{ccccc}
A_{1,1} & A_{1,2} & \cdots & \cdots & A_{1,K} \\
0 & A_{2,2} & \cdots & \cdots & A_{2,K} \\
0 & 0 & \cdots &\ddots& \cdots \\
0 & 0& \cdots & 0& A_{K,K} \\
\end{array}
\right],
\end{equation}
where $A_{i,i}$ is either irreducible or zero, $1\leq i\leq K$. Furthermore,

\begin{equation*}
\varepsilon(A)=\underset{i=1}{\overset{K}{\bigcup}}\varepsilon(A_{i,i}),
\end{equation*}
where $\varepsilon(A)$ and $\varepsilon(A_{i,i})$ are the sets of the eigenvalues of $A$ and $A_{i,i}$, respectively.

Let $n_{i}$ be the size of $A_{i,i}$, $1\leq i \leq K$. For $1\leq i \leq K$, let $m_{i}= \left(\underset{k=1}{\overset{i-1}{\sum}}n_{k}\right)$ and

\begin{equation}\label{eqn:3.16-9}
\mathcal{D}'_{i}=\left\{1\leq j\leq N \hspace{0.1cm}\mid\hspace{0.1cm} m_{i}+1\leq j\leq m_{i+1} \right\}.
\end{equation}

Suppose that there exists $1\leq j \leq K$ such that

\begin{equation}\label{eqn:3.16-109}
\rho(A_{j,j})=\rho(A)>\rho(A_{i,i})
\end{equation}
for $i\neq j$, where $\rho(B)$ is the Perron value of matrix $B$. Then, the natural measure exists and is completely determined by $A_{j,j}$ as in the irreducible cases.

\begin{thm}
\label{theorem:3.18}
Assume that $A=[a_{i,j}]_{ n\times n}$ is reducible of the form (\ref{eqn:3.16-8}), and there exists $1\leq j  \leq k$ such that
\begin{equation}\label{eqn:3.16-110}
\lambda=\rho(A_{j,j})=\rho(A)>\rho(A_{i,i})
\end{equation}
for $i\neq j$. If $A_{j,j}$ is aperiodic, then the natural measure (\ref{eqn:1.6}) exists; if $A_{j,j}$ is irreducible with period $p\geq 2$, then the natural measure (\ref{eqn:1.7}) exists. Moreover,
let $V_{j}=(v_{i})^{t}_{1\leq i \leq n_{j}}$ and $U_{j}=(u_{i})_{1\leq i \leq n_{j}}$ be the right and left eigenvectors of $A_{j,j}$ with respect to $\lambda$. Then, for any finite admissible pattern $[i_{1},i_{2},\cdots, i_{n}]$,

\begin{equation}\label{eqn:3.16-10}
\left\{
\begin{array}{ll}
 \mu([i_{1},i_{2},\cdots, i_{n}])=\frac{u_{i'_{1}}v_{i'_{n}}}{\lambda^{n-1}}a_{i_{1},i_{2}}a_{i_{2},i_{3}}\cdots a_{i_{n-1},i_{n}} & \text {if }i_{1},i_{2},\cdots, i_{n}\in \mathcal{D}'_{j};\\
 & \\
  \mu([i_{1},i_{2},\cdots, i_{n}])=0 & \text { otherwise},
\end{array}
\right.
\end{equation}
where $i'_{1}=i_{1}-m_{j}$ and $i'_{n}=i_{n}-m_{j}$.

\end{thm}

\begin{proof}

Only the proof for irreducible and aperiodic $A$ is presented; the proofs for the other case are analogous.

Suppose that $V'=(v'_{i})^{t}_{1\leq i\leq n}$ and  $U'=(u'_{i})_{1\leq i\leq n}$ are the right and left eigenvectors of $A$ with respect to $\lambda$. For $1\leq k\leq K$, let $V'_{k}=(v'_{k;i})^{t}_{1\leq i\leq n_{k}}$ and $U'_{k}=(u'_{k;i})_{1\leq i\leq n_{k}}$ where

\begin{equation*}
\begin{array}{ccc}
v'_{k;i}=v'_{m_{k}+i} & \text{and} & u'_{k;i}=u'_{m_{k}+i},
\end{array}
\end{equation*}
i.e.,
\begin{equation*}
\begin{array}{ccc}
V'=(V'_{k})^{t}_{1\leq k\leq K} & \text{and} & U'=(U'_{k})_{1\leq k\leq K}.
\end{array}
\end{equation*}
Clearly, from (\ref{eqn:3.16-8}), we have

\begin{equation*}
\left\{
\begin{array}{rcl}
A_{1,1}V'_{1}+A_{1,2}V'_{2}+A_{1,3}V'_{3}+\cdots +A_{1,K}V'_{K} & = & \lambda V'_{1} \\
& & \\
A_{2,2}V'_{2}+A_{2,3}V'_{3}+\cdots +A_{2,K}V'_{K}& = & \lambda V'_{2}\\
&  \vdots &  \\
 A_{K,K}V'_{K}& = & \lambda V'_{K}.
\end{array}
\right.
\end{equation*}
Let $V'_{k}$ be a zero vector for $j+1 \leq k \leq K$. Then, $A_{j,j}V'_{j}=\lambda V'_{j}$, which implies $V'_{j}=V_{j}$. Therefore,
\begin{equation*}
\begin{array}{rl}
 & A_{j-1,j-1}V'_{j-1}+A_{j-1,j}V'_{j}+\cdots +A_{j-1,K}V'_{K} = \lambda V'_{j-1}\\
 & \\
\Rightarrow & (A_{j-1,j-1}-\lambda I)V'_{j-1}= -( A_{j-1,j}V'_{j}+\cdots +A_{j-1,K}V'_{K}).
\end{array}
\end{equation*}
Since $A_{j-1,j-1}-\lambda I$ is invertible, $V'_{j-1}$ can be computed. Subsequently, $V'_{j-1}, \cdots, V'_{1}$ are obtained.

Similarly, let $U'_{k}$ be zero for $1 \leq k \leq j-1$. Then, $U'_{j}A_{j,j}=\lambda U'_{j}$, which implies $U'_{j}=U_{j}$. Hence,
$U'_{j+1}, \cdots, U'_{K}$ can be obtained.

From above result,
\begin{equation*}
\begin{array}{ccc}
 U'V'=1 & \Leftrightarrow& U_{2}V_{2}=1.
 \end{array}
\end{equation*}
As Theorem 3.1, from (\ref{eqn:3.16-8}), (\ref{eqn:3.16-10}) follows. The proof is complete.

\end{proof}

\begin{rmk}
\label{remark:3.20}
When the condition (\ref{eqn:3.16-109}) fails, the finite averages like (\ref{eqn:1.6}) and (\ref{eqn:1.7}) cannot hold. The following example illustrates it.
Let $A=\left[ \begin{array}{cc} 1 & 1 \\ 0 & 1 \end{array}\right]$. Clearly, $A^{n}=\left[ \begin{array}{cc} 1 & n \\ 0 & 1\end{array}\right]$. It is easy to verify that both (\ref{eqn:1.6}) and (\ref{eqn:1.7}) do not exist. Furthermore, for any $M_{1},M_{2}\geq 1$, the limit

\begin{equation}\label{eqn:3.17-100}
\underset{k,l\rightarrow\infty}{\lim} \hspace{0.15cm} \frac{1}{M_{1}+M_{2}+1}  \underset{j=-M_{1}}{\overset{M_{2}}{\sum}} \left|C_{k-j,l+j}([i_{1},i_{2},\cdots, i_{n}]) \right|/ \left|\mathcal{B}_{n+k+l}(\Sigma) \right|
\end{equation}
does not exist, i.e., any finite average can not produce a measure.

However, if (\ref{eqn:3.17-100})  is replaced by
\begin{equation}\label{eqn:3.17-1}
\mu([i_{1},i_{2},\cdots, i_{n}])=\underset{k,l\rightarrow\infty}{\lim} \hspace{0.15cm} \frac{1}{k+l+1}  \underset{j=-l}{\overset{k}{\sum}} \left|C_{k-j,l+j}([i_{1},i_{2},\cdots, i_{n}]) \right|/ \left|\mathcal{B}_{n+k+l}(\Sigma) \right|,
\end{equation}
the measure $\mu$ defined by (\ref{eqn:3.17-1}) exists. Indeed, $\mu(0)=\mu(1)=\frac{1}{2}$ and $\mu(01)=0$. Furthermore, the measure $\mu$ is not ergodic.
\end{rmk}

In the remaining part of this section, the limits (\ref{eqn:1.6}) are (\ref{eqn:1.7}) that are restricted to periodic patterns are considered and they are shown to equal the Parry measure.
Denote by $\mathcal{P}_{n}(\Sigma)$ the set of all admissible periodic patterns (or allowable words) with period $n$:

\begin{equation}\label{eqn:3.18}
\mathcal{P}_{n}(\Sigma)=\left\{x\in \Sigma \hspace{0.1cm}\mid\hspace{0.1cm} \sigma^{n}(x)=x   \right\}.
\end{equation}
For any $k,l\geq 0$, let
\begin{equation}\label{eqn:3.18-1}
C^{(p)}_{k,l}([i_{1},i_{2},\cdots, i_{n}])=C_{k,l}([i_{1},i_{2},\cdots, i_{n}])\bigcap\mathcal{P}_{n+k+l-1}(\Sigma).
\end{equation}
When $A$ is irreducible and aperiodic, the periodic natural measure $\mu^{(p)}$ is defined as

\begin{equation}\label{eqn:3.19}
 \mu^{(p)}([i_{1},i_{2},\cdots, i_{n}])=\underset{k,l\rightarrow\infty}{\lim} \frac{\left|C^{(p)}_{k,l}([i_{1},i_{2},\cdots, i_{n}]) \right|}{\left|\mathcal{P}_{n+k+l-1}(\Sigma) \right|}.
\end{equation}
When $A$ is irreducible with period $p$, the periodic natural measure $\mu^{(p)}$ is defined as
\begin{equation}\label{eqn:3.20}
\mu^{(p)}\left(\left[i_{1},i_{2},\cdots, i_{n} \right]\right)= \underset{n+k+l-1\equiv 0 (\text{ mod } p) }{\underset{k,l\rightarrow\infty}{\lim}} \hspace{0.15cm} \frac{1}{p}\hspace{0.15cm}\frac{  \underset{j=0}{\overset{p-1}{\sum}} \left|C^{(p)}_{k-j,l+j}([i_{1},i_{2},\cdots, i_{n}]) \right|}{   \left|\mathcal{P}_{n+k+l-1}(\Sigma) \right|}.
\end{equation}

The following theorem is obtained in the same way as Theorems 3.1 and 3.3, so the proof is only outlined here.

\begin{thm}
\label{theorem:3.1}
Suppose $A$ is irreducible. Then $\mu^{(p)}$ exists and equals the Parry measure.
\end{thm}

\begin{proof}
Assume $\lambda$ is the Perron value of $A$.
Let $U=(u_{1},u_{2},\cdots,u_{N})$ and $V=(v_{1},v_{2},\cdots,v_{N})^{t}$ be the left and right eigenvectors of $A$ with respect to $\lambda$.

First, the case with irreducible and aperiodic $A$ is considered.
It can be shown that
\begin{equation*}
\left|C^{(p)}_{k,l}([i_{1},i_{2},\cdots, i_{n}]) \right|\sim\left(\underset{q=1}{\overset{N}{\sum}}v_{q}u_{q}\right)\lambda^{k+l}u_{i_{1}}v_{i_{n}}  a_{i_{1},i_{2}}a_{i_{2},i_{3}}\cdots a_{i_{n-1},i_{n}}
\end{equation*}
and
\begin{equation*}
 \left|\mathcal{P}_{n+k+l-1}\right|=\text{tr}\left(A^{n+k+l-1}\right)\sim\left(\underset{q=1}{\overset{N}{\sum}}v_{q}u_{q}\right)\lambda^{n+k+l-1}
\end{equation*}
for large $k$ and $l$. Then,

\begin{equation*}
\begin{array}{rl}
\mu^{(p)}([i_{1},i_{2},\cdots, i_{n}])= & \frac{u_{i_{1}}v_{i_{n}}}{\lambda^{n-1}}a_{i_{1},i_{2}}a_{i_{2},i_{3}}\cdots a_{i_{n-1},i_{n}} \\
& \\
= & \mu([i_{1},i_{2},\cdots, i_{n}])
\end{array}
\end{equation*}
Hence, the result follows immediately.

Now, consider irreducible $A$ with period $p\geq 2$. It can be verified that if $n+k+l-1\equiv 0 (\text{ mod } p)$, then

\begin{equation*}
\underset{j=0}{\overset{p-1}{\sum}} \left|C^{(p)}_{k-j,l+j}([i_{1},i_{2},\cdots, i_{n}]) \right| \sim p^{2}\left(\underset{q=1}{\overset{N}{\sum}}v_{q}u_{q}\right) \lambda^{k+l}u_{i_{1}}v_{i_{n}} a_{i_{1},i_{2}}a_{i_{2},i_{3}}\cdots a_{i_{n-1},i_{n}}.
\end{equation*}
and

\begin{equation*}
 \left|\mathcal{P}_{n+k+l-1}\right|=\text{tr}\left(A^{n+k+l-1}\right)\sim p \left(\underset{q=1}{\overset{N}{\sum}}v_{q}u_{q}\right)\lambda^{n+k+l-1}
\end{equation*}
for large $k$ and $l$. Therefore,
\begin{equation*}
\begin{array}{rl}
\mu^{(p)}([i_{1},i_{2},\cdots, i_{n}])= & \frac{u_{i_{1}}v_{i_{n}}}{\lambda^{n-1}}a_{i_{1},i_{2}}a_{i_{2},i_{3}}\cdots a_{i_{n-1},i_{n}} \\
& \\
= & \mu([i_{1},i_{2},\cdots, i_{n}]).
\end{array}
\end{equation*}
The proof is complete.
\end{proof}

\section{Sofic shifts}
\setcounter{equation}{0}
This section concerns the natural measure of sofic shifts. First, recall the notation and properties of sofic shifts. For details, see Lind and Marcus \cite{1}.

A graph $G$ is composed of a finite set $\mathcal{V}=\mathcal{V}(G)$ of vertices and a finite set $\mathcal{E}=\mathcal{E}(G)$ of edges. Each edge $\xi\in\mathcal{E}$ starts at vertex $i(\xi)\in\mathcal{V}(G) $ and terminates at vertex $t(\xi)\in\mathcal{V}(G)$.

For vertices $i,j\in\mathcal{V}(G)$, denote by $A_{i,j}$ the numbers of edges in $G$ with initial state $i$ and terminal state $j$. The adjacency
matrix of $G$ is $A_{G}=[A_{i,j}]$. Graph $G$ is called irreducible if $A_{G}$ is irreducible. $G$ is called aperiodic if $A_{G}$ is aperiodic.

A labeled graph $\mathcal{G}$ is a pair $\mathcal{G}=(G, \mathcal{L})$ where the labeling $\mathcal{L}:\mathcal{E}\rightarrow \mathcal{S}$ assigns to each edge $\xi\in\mathcal{E}(G)$ a label $\mathcal{L}(\xi)$ from the finite alphabet $\mathcal{S}=\{s_{1},s_{2},\cdots, s_{K}\}$. $\mathcal{G}$ is irreducible if $G$ is irreducible.

For $i,j\in\mathcal{V}(G)$, let $\mathcal{A}_{i,j}$ be the formal sum of alphabets from vertex $i$ to $j$. The symbolic adjacency matrix of $\mathcal{G}$ is defined by $\mathcal{A}=\mathcal{A}(\mathcal{G})=[\mathcal{A}_{i,j}]$.

A labeled graph $\mathcal{G}=(G,\mathcal{L})$ is right-resolving if for each vertex $i$ of $G$, the edges that start at $i$ carry different labels, i.e., for each $i$, the restriction of $\mathcal{L}$ to $\mathcal{E}_{i}$ is one-to-one. A labeled graph $\mathcal{G}$ is called left-resolving if the edges that come to each vertex carry different labels. The right-solving proposition of $\mathcal{G}$ can be expressed in terms of its symbolic adjacency matrix $\mathcal{A}=[\mathcal{A}_{i,j}]$. For each $s\in\mathcal{S}$, denote by

\begin{equation}\label{eqn:4.1}
\mathcal{A}(s)=[\mathcal{A}_{i,j}(s)],
\end{equation}
where $\mathcal{A}_{i,j}(s)$ is $s$ or $\emptyset$. Indeed, $\mathcal{A}_{i,j}(s)=s$ if and only if $s\in \mathcal{A}_{i,j}$. $\mathcal{G}$ is right-resolving if

\begin{equation}\label{eqn:4.2}
\left| \underset{j}{\sum} \mathcal{A}_{i,j}(s)\right|\leq 1
\end{equation}
for any $s\in\mathcal{S}$ and $i\in \mathcal{V}(G)$.
Similarly,  $\mathcal{G}$ is left-resolving if

\begin{equation}\label{eqn:4.2-0}
\left| \underset{i}{\sum} \mathcal{A}_{i,j}(s)\right|\leq 1
\end{equation}
for any $s\in\mathcal{S}$ and $i\in \mathcal{V}(G)$.

First, recall the edge shift $X_{G}$ and the sofic shift $X_{\mathcal{G}}$.
Let $G$ be a graph with edge set $\mathcal{E}$. The edge shift $X_{G}$ is defined by

\begin{equation}\label{eqn:4.2-1}
X_{G}=\left\{\xi=(\xi_{i})_{i=-\infty}^{\infty}\in\mathcal{E}^{\mathbb{Z}^{1}}\hspace{0.1cm}\mid \hspace{0.1cm}t(\xi_{i})=i(\xi_{i+1}) \text{ for all }i\in\mathbb{Z}^{1}\right\}.
\end{equation}
The set of the admissible patterns of $G$ with length $n$ is defined by

\begin{equation}\label{eqn:4.2-2}
\Sigma_{n}(G)=\left\{\xi=(\xi_{i})_{i=0}^{n-1}\in\mathcal{E}^{\mathbb{Z}_{n}}\hspace{0.1cm}\mid \hspace{0.1cm}t(\xi_{i})=i(\xi_{i+1}) \text{ for all }i=0,1,\cdots,n-1\right\}.
\end{equation}

The sofic shift $X_{\mathcal{G}}$ is defined by

\begin{equation}\label{eqn:4.3}
\begin{array}{rl}
X_{\mathcal{G}}= & \left\{x\in\mathcal{S}^{\mathbb{Z}^{1}} \hspace{0.1cm}\mid\hspace{0.1cm} x=\mathcal{L}_{\infty}(\xi)\equiv\left(\mathcal{L}(\xi_{n})\right)_{n=-\infty}^{\infty} \text{ where }
\xi=\left(\xi_{n}\right)_{n=-\infty}^{\infty}\in X_{G} \right\} \\
& \\
= & \mathcal{L}_{\infty}(X_{G}).
\end{array}
\end{equation}
Clearly, $\mathcal{L}$ is a one-block code and $\mathcal{L}_{\infty}: X_{G}\rightarrow X_{\mathcal{G}}$ is a sliding block code. From the definition of $X_{\mathcal{G}}$, $X_{\mathcal{G}}$ is a factor of $X_{G}$.

Based on the subset construction method, every sofic shift has a right-resolving presentation; see Lind and Marcus \cite{1}.

Hereafter, $X_{\mathcal{G}}$ is always assumed to be an irreducible and right-resolving sofic shift with $N$ vertices and $K$ alphabets. Decompose

\begin{equation}\label{eqn:4.4}
\mathcal{A}=\underset{r=1}{\overset{K}{\sum}}\mathcal{A}_{r},
\end{equation}
where $\mathcal{A}_{r}=\mathcal{A}(s_{r})$. Accordingly, the associated adjacency matrix

\begin{equation}\label{eqn:4.5}
\mathbb{A}=\underset{r=1}{\overset{K}{\sum}}\mathbb{A}_{r},
\end{equation}
where $\mathbb{A}_{r}=[a_{r;i,j}]_{N\times N}$. The right-resolving property of $\mathcal{G}$ or condition (\ref{eqn:4.2}) implies

\begin{equation}\label{eqn:4.6}
 \underset{j=1}{\overset{N}{\sum}}a_{r;i,j}\leq 1,
\end{equation}
for all $1\leq r\leq K$.

To compute the number of the set $\mathcal{B}_{n}$ of admissible patterns with length $n$ in $X_{\mathcal{G}}$, some auxiliary graphs, which are similar to those used by Lind and Marcus \cite{1} to study the periodic cycles on zeta functions, are introduced.

First, $\mathcal{G}$ is assumed to be right-resolving. Therefore, for each $i\in\mathcal{V}$, at most one edge that is labeled $s$ starts at $i$. When such an edge exists, the terminal state is denoted as $s(i)$; otherwise $s(i)$ is not defined.

For each $1\leq k\leq N =|\mathcal{V}|$ and $1\leq l\leq k$, the symbolic adjacency matrix  $\widetilde{\mathcal{A}}_{k,l}=\left[\left(\widetilde{\mathcal{A}}_{k,l} \right)_{i,j}\right]$ with the alphabets in $\mathcal{S}$ is defined as follows. Let vertex set $\mathcal{V}_{j}$ be the set of all subsets of $\mathcal{V}$ with $j$ elements. Therefore, $|\mathcal{V}_{j}|=\left(
\begin{array}{c}
N \\ j
\end{array}
\right)$. Fix an order of the elements of $\mathcal{V}_{j}$. Let $\mathcal{V}_{j}=\left\{V_{j;1},V_{j;2},\cdots,V_{j;|\mathcal{V}_{j}|}\right\}$.

Let $V_{k;i}=\{I_{1},I_{2},\cdots, I_{k}\}$ and $V_{l;j}=\{J_{1},J_{2},\cdots, J_{l}\}$ in $\mathcal{V}_{k}$ and $\mathcal{V}_{l}$. Then, $\left(\widetilde{\mathcal{A}}_{k,l} \right)_{i,j}$ equals the formal "sum" of the labels in

\begin{equation}\label{eqn:4.6-2}
\mathcal{S}_{k,l;i,j}=\underset{s\in\mathcal{S}}{\bigcup}\left\{s \hspace{0.1cm}\mid \hspace{0.1cm} \{s(I_{1}),s(I_{2}),\cdots,s(I_{k})\}=V_{l;j}\right\}.
\end{equation}
If $\mathcal{S}_{k,l;i,j}=\emptyset$, $\left(\widetilde{\mathcal{A}}_{k,l} \right)_{i,j}=\emptyset$.
Let $\widetilde{\mathbb{A}}_{k,l}$ be the associated adjacency matrix from $\widetilde{\mathcal{A}}_{k,l}$, which is obtained by setting each $s\in\mathcal{S}$ equal to one. Notably,  $\widetilde{\mathbb{A}}_{1,1}=\mathbb{A}$.

Now, combine $\left\{\widetilde{\mathcal{A}}_{k,l}\right\}_{1\leq k\leq N, 1\leq l\leq k}$ to generate a symbolic matrix $\widehat{\mathcal{A}}$, as follows.
Let

\begin{equation}\label{eqn:3.16-25}
\widehat{\mathcal{A}}=\left[
\begin{array}{ccccc}
\widetilde{\mathcal{A}}_{N,N} & \widetilde{\mathcal{A}}_{N,N-1} & \cdots & \cdots & \widetilde{\mathcal{A}}_{N,1} \\
0 & \widetilde{\mathcal{A}}_{N-1,N-1} & \cdots & \cdots & \widetilde{\mathcal{A}}_{N-1,1} \\
0 & 0 & \cdots &\ddots& \cdots \\
0 & 0& \cdots & 0& \widetilde{\mathcal{A}}_{1,1} \\
\end{array}
\right].
\end{equation}
To express easily the matrix multiple of $\widehat{\mathcal{A}}$, for $n\geq 1$, let

\begin{equation}\label{eqn:3.16-26}
\widehat{\mathcal{A}}^{n}=\left[
\begin{array}{ccccc}
\widehat{\mathcal{A}}_{n;N,N} & \widehat{\mathcal{A}}_{n;N,N-1} & \cdots & \cdots & \widehat{\mathcal{A}}_{n;N,1} \\
0 & \widehat{\mathcal{A}}_{n;N-1,N-1} & \cdots & \cdots & \widehat{\mathcal{A}}_{n;N-1,1} \\
0 & 0 & \cdots &\ddots& \cdots \\
0 & 0& \cdots & 0& \widehat{\mathcal{A}}_{n;1,1} \\
\end{array}
\right],
\end{equation}
where
\begin{equation*}
\widehat{\mathcal{A}}_{n;k,l}=\underset{l\leq i_{n}\leq i_{n-1} \leq\cdots\leq i_{1}\leq k}{\sum}\widetilde{\mathcal{A}}_{k,i_{1}}\widetilde{\mathcal{A}}_{i_{1},i_{2}}\cdots\widetilde{\mathcal{A}}_{i_{n-1},l}
\end{equation*}
for $1\leq k\leq N$ and $1\leq l\leq k$.
The associated adjacency matrix $\widehat{\mathbb{A}}$ of $\widehat{\mathcal{A}}$ can be defined. Similarly, denote by

\begin{equation}\label{eqn:3.16-27}
\widehat{\mathbb{A}}^{n}=\left[
\begin{array}{ccccc}
\widehat{\mathbb{A}}_{n;N,N} & \widehat{\mathbb{A}}_{n;N,N-1} & \cdots & \cdots & \widehat{\mathbb{A}}_{n;N,1} \\
0 & \widehat{\mathbb{A}}_{n;N-1,N-1} & \cdots & \cdots & \widehat{\mathbb{A}}_{n;N-1,1} \\
0 & 0 & \cdots &\ddots& \cdots \\
0 & 0& \cdots & 0& \widehat{\mathbb{A}}_{n;1,1} \\
\end{array}
\right].
\end{equation}

As (\ref{eqn:4.1}), for any $1\leq k\leq N$, $1\leq l\leq k$ and $s\in\mathcal{S}$, define $\widetilde{\mathcal{A}}_{k,l}(s)=\left[\left(\widetilde{\mathcal{A}}_{k,l}(s)\right)_{i,j}\right]$ by
\begin{equation*}
\left(\widetilde{\mathcal{A}}_{k,l}(s)\right)_{i,j}=\left\{
\begin{array}{ll}
s & \text{ if }s\in \left(\widetilde{\mathcal{A}}_{k,l}\right)_{i,j}, \\
& \\
\emptyset & \text{ otherwise.}
\end{array}
\right.
\end{equation*}
Hence, denote by $\widetilde{\mathbb{A}}_{k,l}(s)$ the associated adjacent matrix of $\widetilde{\mathcal{A}}_{k,l}(s)$.
Therefore, $\mathcal{B}_{n}$ can be determined from $\widehat{\mathcal{A}}$ as follows.

\begin{lem}
\label{lemma:4.2}
If $\mathcal{G}$ is right-resolving, then the number of $\mathcal{B}_{n}=\mathcal{B}_{n}(\mathcal{G})$ of all admissible patterns with length $n$ is expressed by

\begin{equation}\label{eqn:4.7-7}
\left|\mathcal{B}_{n}\right|=\underset{k=1}{\overset{N}{\sum}} (-1)^{k+1}\underset{1\leq l\leq k}{\sum} \left|\widehat{\mathbb{A}}_{n;k,l}\right|.
\end{equation}
In particular, if $\mathcal{G}$ is both right-resolving and left-resolving, then
\begin{equation}\label{eqn:4.7-1}
|\mathcal{B}_{n}|=\underset{k=1}{\overset{N}{\sum}} (-1)^{k+1}\left|\widehat{\mathbb{A}}_{n;k,k}\right|.
\end{equation}

\end{lem}

\begin{proof}

Let $u=u_{1}u_{2}\cdots u_{n}\in \mathcal{S}^{n}$. It is easy to verify that $u\in\mathcal{B}_{n}$ if and only if $u$ appears in $\widetilde{\mathcal{A}}_{n;1,1}$. Every $u$ that appears in $\widehat{\mathcal{A}}^{n}$ is in $\mathcal{B}_{n}$.

Given $u=u_{1}u_{2}\cdots u_{n}\in\mathcal{B}_{n}$, define
 \begin{equation*}
\mathcal{L}^{-1}(u)=\left\{\xi=\xi_{1}\xi_{2}\cdots  \xi_{n}\in \Sigma_{n}(G)
\hspace{0.1cm}\mid\hspace{0.1cm} \mathcal{L}(\xi_{1})\mathcal{L}(\xi_{2})\cdots \mathcal{L}(\xi_{n})=u
\right\}.
\end{equation*}
Let $|\mathcal{V}(G)|=N$. Since $\mathcal{G}$ is right-resolving,  $1\leq|\mathcal{L}^{-1}(u)|\leq N$.

Suppose $u=u_{1}u_{2}\cdots u_{n}\in\mathcal{B}_{n}$ with $|\mathcal{L}^{-1}(u)|=j$, $1\leq j\leq N$, it can be easily shown that $u$ totally appears in

\begin{equation*}
\left\{\widehat{\mathcal{A}}_{n;k,l} \hspace{0.1cm}\mid \hspace{0.1cm} 1\leq l\leq k \right\}
\end{equation*}
$\left(
\begin{array}{c}
j \\ k
\end{array}
\right)$ times. Here, $\left(
\begin{array}{c}
j \\ k
\end{array}
\right)=0$ for $k\geq j+1$. Clearly,

\begin{equation*}
\underset{k=1}{\overset{N}{\sum}} (-1)^{k+1}\left(
\begin{array}{c}
j \\ k
\end{array}
\right)=\underset{k=1}{\overset{j}{\sum}} (-1)^{k+1}\left(
\begin{array}{c}
j \\ k
\end{array}
\right) =1.
\end{equation*}
Hence, (\ref{eqn:4.7-7}) follows.

If $\mathcal{G}$ is both right-resolving and left-resolving, for $1\leq l \leq k-1$, it is easy to see that $\left(\widetilde{\mathcal{A}}_{k,l}\right)_{s,t}=\emptyset$ for all $s,t$. Therefore, from (\ref{eqn:4.7-7}), (\ref{eqn:4.7-1}) follows.
%
%
%
\end{proof}

Before the natural measure of sofic shift can be obtained, the following notation must be introduced and some related properties are shown.

For a sofic shift $X_{\mathcal{G}}$, $\mathcal{G'}=(G',\mathcal{L}')$ is called a minimal right-resolving presentation of $X_{\mathcal{G}}$ if it has the fewest vertices of any right-resolving presentation of $X_{\mathcal{G}}$.

Notably, given an irreducible and right-resolving presentation of sofic shift $X_{\mathcal{G}}$, an algorithm for constructing the minimal right-resolving presentation $\mathcal{G'}(G',\mathcal{L}')$ of $X_{g}$ always exists. Moreover, $\mathcal{G'}=(G',\mathcal{L}')$ is irreducible \cite{1}.

The following result is obtained when $\mathcal{G}(G,\mathcal{L})$ is minimal right-resolving and irreducible.

\begin{lem}
\label{lemma:4.3}
Assume $\mathcal{G}=(G,\mathcal{L})$ is right-resolving and irreducible. Let $\lambda_{j}$ be the Perron value of $\widetilde{\mathbb{A}}_{j,j}$. Then

\begin{equation}\label{eqn:4.8}
\begin{array}{ccc}
\lambda_{1}>\lambda_{j}, 2\leq j\leq N & \text{if } & \mathcal{G}=(G,\mathcal{L}) \text{ is minimal right-resolving.}
\end{array}
\end{equation}

\end{lem}
\begin{proof}
Let $\mathcal{G}_{k}$ be the labeled graph that is constructed by $\widetilde{\mathcal{A}}_{k,k}$, $1\leq k\leq N$.
Clearly, if $\mathcal{G}=\mathcal{G}_{1}$ is right-solving, then $\mathcal{G}_{k}$ is right-resolving for all $2\leq k\leq N$. From the construction of $\mathcal{G}_{k}$, $k\geq 2$, $X_{\mathcal{G}_{1}}\supseteq X_{\mathcal{G}_{2}}\supseteq\cdots \supseteq X_{\mathcal{G}_{N}}$.

We prove (\ref{eqn:4.8}) by contradiction.  Assume $\lambda_{1}=\lambda_{2}=\cdots=\lambda_{N}$ or $\lambda_{1}=\lambda_{2}=\cdots=\lambda_{j}>\lambda_{j+1}$ for some $2\leq j\leq N-1$.

Since $\mathcal{G}$ is right-resolving and irreducible, by Theorem 4.3.3 and Corollary 4.4.9 in Lind and Marcus \cite{1}, the assumption implies $X_{\mathcal{G}_{1}}=X_{\mathcal{G}_{2}}=\cdots=X_{\mathcal{G}_{N}}$ or $X_{\mathcal{G}_{1}}=X_{\mathcal{G}_{2}}=\cdots=X_{\mathcal{G}_{j}}\supsetneqq X_{\mathcal{G}_{j+1}}$.

\item[(i)] Consider the case of $X_{\mathcal{G}_{1}}=X_{\mathcal{G}_{2}}=\cdots=X_{\mathcal{G}_{N}}$. It is easy to see that the number of the vertices in $\mathcal{G}_{k}$ is $\left(
\begin{array}{c}
N \\ k
\end{array}
\right)$. Since the number of the vertices in $\mathcal{G}_{N}$ is $1<N$, $\mathcal{G}$ is not minimal right-resolving. This leads a contradiction.

\item[(ii)] Consider the case of $X_{\mathcal{G}_{1}}=X_{\mathcal{G}_{2}}=\cdots=X_{\mathcal{G}_{j}}\supsetneqq X_{\mathcal{G}_{j+1}}$. It is easy to see that for any admissible pattern $u=u_{1}u_{2}\cdots u_{n}$ with length $n$, $\left|\mathcal{L}^{-1}(u) \right|\geq j$. Moreover, there exists an admissible pattern $y=y_{1}y_{2}\cdots y_{m}$ such that $\left|\mathcal{L}^{-1}(u) \right|= j$. Then, $y=y_{1}y_{2}\cdots y_{m}$ appears in $\widetilde{\mathcal{A}}_{j,j}^{n}$ once.

     By merging the vertices of $\mathcal{G}_{j}$ with the same follower set \cite{1}, the minimal right-resolving labeled graph $\mathcal{G}'=(G',\mathcal{L}')$ of $\mathcal{G}_{j}$ can be constructed. Clearly, $y=y_{1}y_{2}\cdots y_{m}$ appears in $\widetilde{\mathcal{A}}^{n}_{1,1}(\mathcal{G}')$ once. Since $X_{\mathcal{G}_{1}}=X_{\mathcal{G}_{j}}$, $\mathcal{G}_{j}'$ is the minimal right-resolving labeled graph of $X_{\mathcal{G}}$. However, $y$ appears in $\widetilde{\mathcal{A}}_{1,1}^{n}(\mathcal{G})$ $j$ times. Clearly, $\mathcal{G}$ and $\mathcal{G}'$ are different under isomorphism. Lind and Marcus \cite{1} demonstrated that the minimal right-resolving label graph of an irreducible sofic shift is unique under isomorphism. Hence, $ \mathcal{G}=(G,\mathcal{L})$ is not minimal right-resolving and a contradiction is obtained.

Therefore, (\ref{eqn:4.8}) follows. The proof is complete.

\end{proof}

Now, the natural measure of sofic shift $\mathcal{G}(G,\mathcal{L})$ can be obtained as follows.
Recall

\begin{equation*}
\begin{array}{rl}
& C_{k,l}([s_{i_{1}},s_{i_{2}},\cdots,s_{i_{n}} ]) \\
 & \\
 = & \left\{ x=(x_{-k+1},\cdots,x_{-1},x_{0},x_{1},\cdots,x_{n+l} )\in  \mathcal{B} _{n+k+l}(X_{\mathcal{G}})\hspace{0.1cm}\mid\hspace{0.1cm}x_{j}=s_{i_{j}}. 1\leq j\leq n \right\}
\end{array}
\end{equation*}
$\mathcal{B}_{n}$ is the set of all admissible patterns with length $n$ in $X_{\mathcal{G}}$. Equation (\ref{eqn:4.8}) yields the following results.

\begin{thm}
\label{theorem:4.4}
Assume $\mathcal{G}$ is minimal right-resolving. Then,

\item[(i)] if $\mathcal{G}$ is irreducible and aperiodic, then
\begin{equation}\label{eqn:4.8-0}
\begin{array}{rl}
& \underset{k,l\rightarrow\infty}{\lim} \left|C_{k,l}([s_{i_{1}},s_{i_{2}},\cdots,s_{i_{n}} ])\right|/|\mathcal{B} _{n+k+l}| \\
& \\
= & \frac{1}{\lambda^{n}}\underset{i}{\sum}\underset{j}{\sum} u_{i} (\mathbb{A}_{i_{1}}\mathbb{A}_{i_{2}}\cdots\mathbb{A}_{i_{n}})_{i,j}v_{j},
\end{array}
\end{equation}
where $\lambda$ is the Perron value of $\mathbb{A}$, and $V=(v_{1},v_{2},\cdots,v_{N})^{t}$ and  $U=(u_{1},u_{2},\cdots,u_{N})$ are the right and left eigenvectors with respect to $\lambda$, normalized such that $UV=1$. Here, $\mathbb{A}_{r}$ is defined as in (\ref{eqn:4.5}).

\item[(ii)] if $\mathcal{G}$ is irreducible with period $p\geq 2$,

\begin{equation}\label{eqn:4.8-3}
\begin{array}{rl}
 & \underset{k,l\rightarrow\infty}{\lim} \hspace{0.15cm} \frac{1}{p} \underset{j=1}{\overset{p}{\sum}} \left|C_{k-j,l+j}([s_{i_{1}},s_{i_{2}},\cdots,s_{i_{n}}]) \right|/ \left|\mathcal{B}_{n+k+l}(\Sigma) \right| \\
 & \\
 = & \frac{1}{\lambda^{n}}\underset{i}{\sum}\underset{j}{\sum} u_{i} (\mathbb{A}_{i_{1}}\mathbb{A}_{i_{2}}\cdots\mathbb{A}_{i_{n}})_{i,j}v_{j},
\end{array}
\end{equation}
where $\lambda$ is the Perron value of $\mathbb{A}$.
\end{thm}

\begin{proof}
(i). From Lemma 4.2, it can be obtained that

\begin{equation}\label{eqn:4.8-4}
|\mathcal{B}_{n+k+l}|=\underset{i=1}{\overset{N}{\sum}} (-1)^{i+1}\underset{1\leq q\leq i}{\sum} \left|\widehat{\mathbb{A}}_{n+k+l;i,q}\right|
\end{equation}
and

\begin{equation}\label{eqn:4.8-5}
\begin{array}{rl}
 & \left|C_{k,l}([s_{i_{1}},s_{i_{2}},\cdots,s_{i_{n}}]) \right| \\
 & \\
 = & \underset{i=1}{\overset{N}{\sum}} (-1)^{i+1}\underset{1\leq q_{n+2} \leq q_{n+1} \leq \cdots \leq q_{1}\leq i}{\sum}\left|\widehat{\mathbb{A}}_{k;i,q_{1}}\widetilde{\mathbb{A}}_{q_{1},q_{2}}(s_{i_{1}})
\widetilde{\mathbb{A}}_{q_{2},q_{3}}(s_{i_{2}}) \cdots \widetilde{\mathbb{A}}_{q_{n},q_{n+1}}(s_{i_{n}})
 \widehat{\mathbb{A}}_{l;q_{n+1},q_{n+2}}
 \right|.
 \end{array}
\end{equation}

From Lemma 4.2, the Perron value $\lambda$ of $\mathbb{A}=\widetilde{\mathbb{A}}_{1,1}$ is strictly larger than those of $\widetilde{\mathbb{A}}_{j,j}$, $2\leq j\leq N$. As the proof of Theorem 3.5, it can be verified that

\begin{equation*}
|\mathcal{B}_{n+k+l}|\sim \lambda^{n+k+l} (v_{1}+v_{2}+\cdots +v_{N}) (u_{1}+u_{2}+\cdots +u_{N})
\end{equation*}
and
\begin{equation*}
\begin{array}{rl}
 & \left|C_{k,l}([s_{i_{1}},s_{i_{2}},\cdots,s_{i_{n}}]) \right| \\
 & \\
 \sim & \lambda^{k+l} (v_{1}+v_{2}+\cdots +v_{N})\left( \underset{i}{\sum}\underset{j}{\sum} u_{i} (\mathbb{A}_{i_{1}}\mathbb{A}_{i_{2}}\cdots\mathbb{A}_{i_{n}})_{i,j}v_{j}\right) (u_{1}+u_{2}+\cdots +u_{N})
\end{array}
\end{equation*}
Therefore, (\ref{eqn:4.8-0}) follows.

(ii)  As in (i), the result (\ref{eqn:4.8-3}) can be obtained. The detail is omitted.
\end{proof}

Therefore, the natural measure of irreducible sofic shift $\mathcal{G}=(G,\mathcal{L})$ is defined as follows.

\begin{defi}
\label{definition:4.5}
Suppose $\mathcal{G}(G,\mathcal{L})$ is irreducible and right-resolving with $K$ alphabets. The natural measure $\mu$ of $\mathcal{G}(G,\mathcal{L})$ is defined as follows:
for any finite admissible pattern $s_{i_{1}}s_{i_{2}}\cdots s_{i_{n}}$, $n\geq 1$,

\begin{equation}\label{eqn:4.8-1}
\mu([s_{i_{1}}s_{i_{2}}\cdots s_{i_{n}}])\equiv\frac{1}{\lambda^{n}}\underset{i}{\sum}\underset{j}{\sum} u_{i} (\mathbb{A}_{i_{1}}\mathbb{A}_{i_{2}}\cdots\mathbb{A}_{i_{n}})_{i,j}v_{j},
\end{equation}
where $\lambda$ is the Perron value of $\mathbb{A}$, and $\mathbb{A}_{k}$ is defined as in (\ref{eqn:4.5}), $1\leq k\leq K$.
\end{defi}
 Notably, the definition (\ref{eqn:4.8-1}) requires only that $\mathcal{G}=(G,\mathcal{L})$ is irreducible and right-resolving: it need not be minimal. In Theorem 4.9, the uniqueness of the natural measure will justify Definition 4.4.

The natural measure of irreducible sofic shift can also be obtained as follows.
Let
\begin{equation}\label{eqn:4.6-3}
\begin{array}{rl}
& \widehat{C}_{k,l}([s_{i_{1}},s_{i_{2}},\cdots,s_{i_{n}} ]) \\
 & \\
 = & \left\{ x=(x_{-k+1},\cdots,x_{-1},x_{0},x_{1},\cdots,x_{n+l} )\in  \mathcal{B} _{n+k+l}(X_{G})\hspace{0.1cm}\mid\hspace{0.1cm}\mathcal{L}(x_{j})=s_{i_{j}}, 1\leq j\leq n \right\}.
\end{array}
\end{equation}
It is easy to verify that

\begin{equation}\label{eqn:4.6-4}
\left|\widehat{C}_{k,l}([s_{i_{1}},s_{i_{2}},\cdots,s_{i_{n}} ]) \right|
 = \left|\mathbb{A}^{k}\mathbb{A}_{i_{1}}\mathbb{A}_{i_{2}}\cdots \mathbb{A}_{i_{n}}\mathbb{A}^{k}\right|
\end{equation}
and

\begin{equation}\label{eqn:4.6-5}
\left|\mathcal{B}_{n+k+l}(X_{G})\right|=\left|\mathbb{A}^{n+k+l}\right|.
\end{equation}
Notably, by comparing with (\ref{eqn:4.8-4}) and (\ref{eqn:4.8-5}), $\left|\widehat{C}_{k,l}([s_{i_{1}},s_{i_{2}},\cdots,s_{i_{n}}])\right|$ and $\left|\mathcal{B}_{n+k+l}(X_{G})\right|$ are much easier in computation.

Then, if $\mathbb{A}$ is irreducible and aperiodic, then the measure $\widehat{\mu}$ is defined by

\begin{equation}\label{eqn:4.6-6}
\begin{array}{rl}
 & \widehat{\mu}([s_{i_{1}},s_{i_{2}},\cdots, s_{i_{n}}]) \\
 & \\
=  & \underset{k,l\rightarrow\infty}{\lim} \left|\widehat{C}_{k,l}([i_{1},i_{2},\cdots, i_{n}]) \right|/ \left|\mathcal{B}_{n+k+l}(X_{G}) \right|;
\end{array}
\end{equation}
if $\mathbb{A}$ is irreducible with period $p\geq2$,

\begin{equation}\label{eqn:4.6-7}
 \widehat{\mu}\left(\left[s_{i_{1}},s_{i_{2}},\cdots, s_{i_{n}} \right]\right)=\underset{k,l\rightarrow\infty}{\lim} \hspace{0.15cm} \frac{1}{p} \underset{j=0}{\overset{p-1}{\sum}} \left|\widehat{C}_{k-j,l+j}([i_{1},i_{2},\cdots, i_{n}]) \right|/ \left|\mathcal{B}_{n+k+l}(X_{G}) \right|.
\end{equation}

As the proofs of Theorems 3.1 and 3.3, the following result can be obtained immediately. The details are omitted.
\begin{cor}
\label{corollary:4.1-1}
Assume $\mathcal{G}$ is right-resolving and irreducible. Then, $\widehat{\mu}=\mu$

\end{cor}

Now, consider the relation between the natural measure and the hidden Markov measure (or sofic measure). The hidden Markov measure is introduced as follows \cite{85}.
Suppose $(X,\sigma)$ and $(Y,\sigma)$ are shift spaces and $\pi : X\rightarrow Y$ is a sliding block code (factor map).  Then, for each $m\in M(X)$, $\pi$ induces a measure $\pi^{*} m\in M(Y)$ by

\begin{equation}\label{eqn:4.14}
\pi^{*} m (E) \equiv m (\pi^{-1}E)
\end{equation}
for all measurable $E\subset Y$.

Suppose $X$ is a shift of finite type and $m$ is a Markov measure on $X$. If $\pi : X\rightarrow Y$ is a sliding block code, then $\pi^{*}m$ is called a hidden Markov measure. Notably, a shift space is sofic if and only if it is a factor of a shift of finite type \cite{1}.

The following theorem shows that the natural measure of an irreducible sofic shift equals the hidden Markov measure of underlying shifts of finite type.

\begin{thm}
\label{theorem:4.4}
Let $\mu_{G}$ and $\mu_{\mathcal{G}}$ be the natural measured of $X_{G}$ and $X_{\mathcal{G}}$ respectively. Assume $\mathcal{G}=(G,\mathcal{L})$ is irreducible and minimal right-resolving. Then, $\mu_{\mathcal{G}}= \pi^{*}\mu_{G}=\mathcal{L}_{\infty}^{*}\mu_{G}$.
\end{thm}

\begin{proof}
Let $\mathcal{G}=(G,\mathcal{L})$ be the irreducible and minimal right-resolving presentation of sofic shift $Y$ with $K$ alphabets. Let $\mathbb{A}=\underset{r=1}{\overset{K}{\sum}} \mathbb{A}_{r}$ be the adjacency matrix of $\mathcal{G}$. Clearly, $A(G)=\mathbb{A}$ is irreducible. Then, $\Sigma(\mathbb{A})=X_{G}$.

From (\ref{eqn:4.6-3}) and (\ref{eqn:4.6-6}), it is easy to verified that $\widehat{\mu}_{\mathcal{G}}= \pi^{*}\mu_{G}$. Therefore, by Corollary 4.5, the result is obtained.
%
%
%

\end{proof}

Now, the natural measure $\mu_{\mathcal{G}}$ of sofic shift will be shown to be the only measure with ergodicity and maximal entropy.
The following lemma is established first.

\begin{lem}
\label{lemma:4.1}
If $\mathcal{G}=(G,\mathcal{L})$ is right-resolving, then

\begin{equation}\label{eqn:4.6-1}
\left|\underset{n=1}{\overset{\infty}{\bigcup}} \hspace{0.2cm} \underset{i_{k}\in\{1,2,\cdots,r\}}{\bigcup}\left\{\mathbb{A}_{i_{1}}\mathbb{A}_{i_{2}}\cdots \mathbb{A}_{i_{n}}\right\}\right|\leq N^{N}.
\end{equation}
\end{lem}

\begin{proof}
The result can be obtained by the following observation.

Let $C=AB$, where $C=[c_{i,j}]$,  $A=[a_{i,j}]$ and $B=[b_{i,j}]$ are
$0-1$ matrices. If

\begin{equation*}
\begin{array}{ccc}
\underset{j}{\sum}a_{i,j}\leq 1 & \text{and } & \underset{j}{\sum}b_{i,j}\leq 1
\end{array}
\end{equation*}
for all $i$, then
\begin{equation*}
\underset{j}{\sum}c_{i,j}\leq 1 .
\end{equation*}
Indeed,

\begin{equation*}
\underset{j}{\sum}c_{i,j} =\underset{j}{\sum}\underset{k}{\sum}a_{i,k}b_{k,j}.
\end{equation*}
If $\underset{j}{\sum}a_{i,k}=1$, then $a_{i,k'}=1$ and $a_{i,k''}=0$ for some $k'$ and for all $k''\neq k'$. Hence,
\begin{equation*}
\underset{j}{\sum}c_{i,j}=\underset{j}{\sum}b_{k',j}\leq 1.
\end{equation*}

The total number of members of the set of all matrices with entries $0$ and $1$ that satisfy (\ref{eqn:4.6}) is $N^{N}$, which yields the result.
\end{proof}

Now, Lemma 4.7 implies that the measure defined in (\ref{eqn:4.8-1}) is uniformly distributed as follows.

\begin{lem}
\label{lemma:4.5}
There exist constants $\beta\geq\alpha > 0$ such that

\begin{equation}\label{eqn:4.8-2}
\frac{\alpha}{\lambda^{n}}\leq \mu\left([s_{i_{1}},s_{i_{2}},\cdots,s_{i_{n}}]\right)\leq\frac{\beta}{\lambda^{n}}
\end{equation}
for any admissible pattern $[s_{i_{1}},s_{i_{2}},\cdots,s_{i_{n}}]$.

\end{lem}

We now prove the following main result for irreducible sofic shifts.

\begin{thm}
\label{theroem:4.6}
If $\mathcal{G}=(G,\mathcal{L})$ is an irreducible and right-resolving, then the natural measure defined by (\ref{eqn:4.8-1}) is the only measure with ergodicity and maximal entropy. Furthermore, if $\mathcal{G}$ is aperiodic, then the natural measure is strong mixing.
\end{thm}

\begin{proof}
The proof is similar to that of the Parry measure for a shift of finite type \cite{2}. For brevity, only the essential parts of the proof are presented here.

\item[(i)] First, $\mu$ is proven to be strong mixing when $\mathcal{G}=(G,\mathcal{L})$ is aperiodic; for a shift of finite type, see Theorem 1.19 in Walters \cite{2}. It suffices to prove that

\begin{equation}\label{eqn:4.10}
\underset{k\rightarrow \infty}{\lim}\mu\left(\sigma^{-k}(E)\cap F\right)=\mu(E)\mu(F)
\end{equation}
for any $E$ and $F$ in $\mathfrak{B}(X_{\mathcal{G}})$. Indeed, only (\ref{eqn:4.10}) for any two cylinders has to be shown. Let

\begin{equation*}
\begin{array}{ccc}
A=C([a;s_{i_{1}},s_{i_{2}},\cdots,s_{i_{p}}]) & \text{and} & B=C([b;s_{j_{1}},s_{j_{2}},\cdots,s_{j_{q}}]),
\end{array}
\end{equation*}
where $a,b\in\mathbb{Z}^{1}$ are starting points of cylinders $A$ and $B$, respectively. Notably,
\begin{equation*}
\sigma^{-k}(A)=C([a+k;s_{i_{1}},s_{i_{2}},\cdots,s_{i_{p}}]).
\end{equation*}

For large $k$,
\begin{equation*}
\begin{array}{rl}
 & \mu\left(\sigma^{-k}(A)\cap B\right)\\
 & \\
 = & \frac{1}{\lambda^{p+q+k}}\underset{i,j}{\sum}\left[u_{i}\left(\mathbb{A}_{j_{1}}\mathbb{A}_{j_{2}}\cdots\mathbb{A}_{j_{q}}\mathbb{A}^{k}
 \mathbb{A}_{i_{1}}\mathbb{A}_{i_{2}}\cdots\mathbb{A}_{i_{p}}\right)_{i,j}v_{j}\right]
\end{array}
\end{equation*}
By (\ref{eqn:2.1-4}) of Proposition 2.5,

\begin{equation*}
\begin{array}{rl}
 & \underset{k\rightarrow \infty}{\lim}\mu\left(\sigma^{-k}(A)\cap B\right) \\
 & \\
 =&  \frac{1}{\lambda^{p+q}}\left(\underset{i}{\sum}\underset{\alpha}{\sum}u_{i}(\mathbb{A}_{j_{1}}\mathbb{A}_{j_{2}}\cdots\mathbb{A}_{j_{q}})_{i,\alpha}v_{\alpha}\right)
\left( \underset{j}{\sum}\underset{\beta}{\sum}u_{\beta}(\mathbb{A}_{i_{1}}\mathbb{A}_{i_{2}}\cdots\mathbb{A}_{i_{p}})_{\beta,j}v_{j} \right)\\
 & \\
 = & \mu(A)\mu(B).
\end{array}
\end{equation*}

If $\mathcal{G}(G,\mathcal{L})$ is irreducible with period $p\geq 2$, then $\mu$ is proven to be ergodic. It suffices to show that

\begin{equation}\label{eqn:4.11}
\underset{K\rightarrow \infty}{\lim}\frac{1}{K}\underset{k=0}{\overset{K-1}{\sum}}\mu\left(\sigma^{-k}(E)\cap F\right)=\mu(E)\mu(F).
\end{equation}
Now, (\ref{eqn:2.1-2}) is needed. Based on the above argument, the ergodicity of $\mu$ can be proven. The details are omitted here.

\item[(ii)] Next, $\mu$ is shown to have maximum entropy of $\log \lambda$ when $\mathcal{G}(G,\mathcal{L})$ is irreducible and aperiodic; for a shift of finite type, see Theorem 8.10 of Walters \cite{2}. It suffices to show that

\begin{equation}\label{eqn:4.11-1}
\begin{array}{rl}
 & \underset{n\rightarrow\infty}{\lim}\left\{-\frac{1}{n}\underset{[s_{i_{1}},s_{i_{2}},\cdots,s_{i_{n}}]\in \Sigma_{n}(X_{\mathcal{G}})}{\sum}
\mu\left([s_{i_{1}},s_{i_{2}},\cdots,s_{i_{n}}]\right)\log \mu\left([s_{i_{1}},s_{i_{2}},\cdots,s_{i_{n}}]\right)\right\}\\
& \\ = & \log \lambda.
\end{array}
\end{equation}
By Lemma 4.8, the left-hand side of (\ref{eqn:4.11-1}) can be proven to equal

\begin{equation*}
\begin{array}{rl}
 &  \underset{n\rightarrow\infty}{\lim}\left\{-\frac{1}{n}\underset{[s_{i_{1}},s_{i_{2}},\cdots,s_{i_{n}}]\in \Sigma_{n}(X_{\mathcal{G}})}{\sum}
\mu\left([s_{i_{1}},s_{i_{2}},\cdots,s_{i_{n}}]\right)\log \frac{1}{\lambda^{n}}\right\} \\
& \\
= & \log \lambda  \left\{\underset{n\rightarrow\infty}{\lim}\underset{[s_{i_{1}},s_{i_{2}},\cdots,s_{i_{n}}]\in \Sigma_{n}(X_{\mathcal{G}})}{\sum}
\mu\left([s_{i_{1}},s_{i_{2}},\cdots,s_{i_{n}}]\right)\right\}\\
& \\
 = & \log \lambda .
\end{array}
\end{equation*}

\item[(iii)] Finally, $\mu$ is shown to be the unique probability measure that is ergodic and has maximum entropy. The argument is similar to the proof of Theorem 8.10 in Walters \cite{2}, in which the Parry measure is proven to be the unique measure with ergodicity and maximum entropy. The only different element of the sofic shift is the uniform distribution property (\ref{eqn:4.8-2}) of Lemma 4.8. The uniform distribution property for shifts of finite type immediately follows from (\ref{eqn:2.1-3}) and (\ref{eqn:2.2}). The details are omitted for brevity.
\end{proof}

The rest of this section considers the natural measure of periodic patterns of sofic shift as for shifts of finite type. Firstly, recall the results concerning the presentation of periodic patterns in $X_{\mathcal{G}}$ \cite{1}. Let $\mathcal{P}_{n}(\mathcal{G})$ be the set of periodic patterns with period $n$ in $X_{\mathcal{G}}$. A periodic pattern $U_{n}=(u_{1},u_{2},\cdots, u_{n})^{\infty}$ with period $n$ can appear at the sites $(k,l)$ of $\widetilde{\mathcal{A}}^{n}_{i,j}$ for $k\neq l$. Then, $\text{tr}\left(\widetilde{\mathcal{A}}^{n}_{j,j}\right)$, $1\leq j\leq N$, does not suffice to present $\mathcal{P}_{n}(\mathcal{G})$.
Lind and Marcus \cite{1} introduced signed symbolic matrices $\widetilde{\mathcal{B}}_{j}$, $1\leq j\leq N$, to present $\mathcal{P}_{n}(\mathcal{G})$ as follows. Let vertex set $\mathcal{V}_{j}$ be the set of all subsets of $\mathcal{V}$ with $j$ elements. Fix an order of the states in each element of $\mathcal{V}_{j}$. Let $V_{j;k}=\{I_{1},I_{2},\cdots, I_{j}\}$ and $V_{j;l}=\{J_{1},J_{2},\cdots, J_{j}\}$ be in $\mathcal{V}_{j}$. Let $s\in\mathcal{S}$. If $(s(I_{1}),s(I_{2}),\cdots,s(I_{j}))$ is an even permutation of $V_{j;l}$, then $\left(\widetilde{\mathcal{B}}_{j}\right)_{k,l}$ contains the label $s$. If the permutation is odd, then $\left(\widetilde{\mathcal{B}}_{j}\right)_{k,l}$ contains the signed label $-s$. Each entry in $\widetilde{\mathcal{B}}_{j}$ is a signed combination of labels in $\mathcal{S}$. Define the associated transition matrix  $\widetilde{\mathbb{B}}_{j}$ by setting all labels in $\mathcal{S}$ to $1$. Notably, $\widetilde{\mathbb{B}}_{1}=\widetilde{\mathbb{A}}_{1,1}=\mathbb{A}$.  Then, the following lemma follows \cite{1}.

\begin{lem}
\label{lemma:4.2-2}
Let $\mathcal{G}$ be a right-resolving labeled graph with $N$ vertices. Then

\begin{equation}\label{eqn:4.7-4}
\mathcal{P}_{n}(X_{\mathcal{G}})=\underset{j=1}{\overset{N}{\sum}}(-1)^{j+1}\text{tr}\left(\widetilde{\mathcal{B}}^{n}_{j}\right)
\end{equation}
and
\begin{equation}\label{eqn:4.7-5}
\left|\mathcal{P}_{n}(X_{\mathcal{G}})\right|=\underset{j=1}{\overset{N}{\sum}}(-1)^{j+1}\text{tr}\left(\widetilde{\mathbb{B}}^{n}_{j}\right).
\end{equation}
\end{lem}

As (\ref{eqn:3.18-1}) $\sim$ (\ref{eqn:3.20}), for any $k,l\geq 0$, let
\begin{equation*}
C^{(p)}_{k,l}([s_{i_{1}},s_{i_{2}},\cdots, s_{i_{n}}]=C_{k,l}([s_{i_{1}},s_{i_{2}},\cdots, s_{i_{n}}])\bigcap\mathcal{P}_{n+k+l-1}(X_{\mathcal{G}}).
\end{equation*}
When $\mathbb{A}$ is irreducible and aperiodic, define the periodic natural measure $\mu^{(p)}$ by

\begin{equation}\label{eqn:4.20}
 \mu^{(p)}([s_{i_{1}},s_{i_{2}},\cdots, s_{i_{n}}])=\underset{k,l\rightarrow\infty}{\lim} \frac{\left|C^{(p)}_{k,l}([s_{i_{1}},s_{i_{2}},\cdots, s_{i_{n}}]) \right|}{\left|\mathcal{P}_{n+k+l-1}(X_{\mathcal{G}}) \right|}.
\end{equation}
When $\mathbb{A}$ is irreducible with period $p$, define the periodic natural measure $\mu^{(p)}$ by
\begin{equation}\label{eqn:4.21}
\mu^{(p)}\left(\left[s_{i_{1}},s_{i_{2}},\cdots, s_{i_{n}} \right]\right)= \underset{n+k+l-1\equiv 0\hspace{0.1cm} (\text{mod } p)}{\underset{k,l\rightarrow\infty}{\lim}} \hspace{0.15cm} \frac{1}{p}\hspace{0.15cm}\frac{ \underset{j=0}{\overset{p-1}{\sum}} \left|C^{(p)}_{k-j,l+j}([s_{i_{1}},s_{i_{2}},\cdots, s_{i_{n}}]) \right|}{ \left|\mathcal{P}_{n+k+l-1}(X_{\mathcal{G}}) \right|}.
\end{equation}

Now, the periodic natural measure is established to equal the natural measure.

\begin{thm}
\label{theorem:4.4}
Assume $\mathcal{G}$ is minimal right-resolving and irreducible. Then,
$\mu^{(p)}_{\mathcal{G}}=\mu_{\mathcal{G}}$.
\end{thm}

\begin{proof}
Let $\lambda_{j}$ be the Perron value of $\widetilde{\mathbb{A}}_{j,j}$ for $1\leq j\leq N$.
From Lemma 4.3, $\lambda_{1}>\lambda_{i}$, $2\leq i\leq N$.

From the construction of $\widetilde{\mathcal{B}}_{j}$, $1\leq j\leq N$, we have $\widetilde{\mathbb{B}}_{1}=\widetilde{\mathbb{A}}_{1,1}$ and
 $-\widetilde{\mathbb{A}}_{i,i}\leq \widetilde{\mathbb{B}}_{i}\leq \widetilde{\mathbb{A}}_{i,i}$ for $i\geq 2$.
It can be shown that every eigenvalue of $\widetilde{\mathbb{B}}_{i}$ is in $[-\lambda_{i},\lambda_{i}]$, $i\geq 2$.
Therefore, as the proof of Theorem 3.6, the result follows immediately.
\end{proof}

\section{Countable state shifts}
\setcounter{equation}{0}
This section study the natural measure for both countable state shifts of finite type and sofic shifts.

First, some notation and results concerning countable state shifts are recalled from Kitchens \cite{1-1}.
A matrix is called countable if its rows and columns are indexed by the same set $\mathcal{I}$ that is either finite or countably infinite. In this section, the matrix $T$ is always assumed to be countably infinite. As for finite matrices, a countable non-negative matrix $T=[T_{i,j}]$ is irreducible if for every pair of indices $i$ and $j$, there exists $l$ such that $(T^{l})_{i,j}>0$. For any index $i$, the period of $i$ is defined by $p(i)=\gcd\left\{l \hspace{0.1cm} \mid\hspace{0.1cm}  (T^{l})_{i,i}>0\right\}$. Notably, when $T$ is irreducible, $p(i)$ is independent of $i$ and is called the period of $T$. If $T$ is periodic with period one, $T$ is said to be aperiodic.

When $T$ is countable, non-negative, irreducible and aperiodic, for any pair of indices $i$ and $j$, the value $\underset{n\rightarrow \infty}{\lim} \sqrt[n]{(T^{n})_{i,j}}$ can be demonstrated to exist and be independent of the pair of indices $i$ and $j$. The Perron value of $T$ is defined by $\lambda=\underset{n\rightarrow \infty}{\lim} \sqrt[n]{(T^{n})_{i,j}}$. In this work, $\lambda$ is always assumed to be finite.

Before the generalized Perron-Frobenious Theorem for a countable non-negative matrix can be introduced, more definitions must be presented.
For any pair of indices $i$ and $j$, define

\begin{equation*}
\left\{
\begin{array}{llll}
t_{i,j}(0)=\delta_{i,j}, & t_{i,j}(1)=T_{i,j} & \text{and} & t_{i,j}(n)= (T^{n})_{i,j}, \\
& & & \\
l_{i,j}(0)=0, & l_{i,j}(1)=T_{i,j} & \text{and}& l_{i,j}(n+1)= \underset{r\neq i}{\sum}l_{i,r}(n)t_{r,j},\\
& & & \\
r_{i,j}(0)=0, & r_{i,j}(1)=T_{i,j} & \text{and}& r_{i,j}(n+1)= \underset{r\neq j}{\sum}t_{i,r}r_{r,j}(n).
\end{array}
\right.
\end{equation*}
The three generating functions are defined by
\begin{equation*}
\left\{
\begin{array}{l}
T_{i,j}(z)=\underset{n=1}{\overset{\infty}{\sum}}t_{i,j}(n)z^{n},\\
\\
L_{i,j}(z)=\underset{n=1}{\overset{\infty}{\sum}}l_{i,j}(n)z^{n},\\
\\
R_{i,j}(z)=\underset{n=1}{\overset{\infty}{\sum}}r_{i,j}(n)z^{n}.\\
\end{array}
\right.
\end{equation*}
Clearly, the radius of convergence of every $T_{i,j}(z)$ is $1/\lambda$.
$T$ is said to be recurrent if $T_{i,i}(1/\lambda) =\infty$ and transient if $T_{i,i}(1/\lambda) <\infty$. If $T$ is recurrent, $T$ is positive recurrent if
\begin{equation*}
 \underset{n=1}{\overset{\infty}{\sum}}\frac{n l_{i,i}(n)}{\lambda^{n}}< \infty
\end{equation*}
and null recurrent if
\begin{equation*}
 \underset{n=1}{\overset{\infty}{\sum}}\frac{n l_{i,i}(n)}{\lambda^{n}}= \infty.
\end{equation*}
Notably, these two definitions are independent of the choice of $i$.

Now, the generalized Perron-Frobenius Theorem is presented as follows; see Theorem 7.1.3 in Kitchens \cite{1-1}.

\begin{thm}
\label{theorem:5.1}[Generalized Perron-Frobenius Theorem]
Suppose $T$ is countable, non-negative, irreducible, aperiodic and recurrent. Then there exists a Perron value $\lambda>0$ (assumed to be finite) such that:

\item[(i)] $\lambda=\underset{n\rightarrow \infty}{\lim} \sqrt[n]{(T^{n})_{i,j}}$ for any indices $i$ and $j$;

\item[(ii)] $\lambda$ has strictly positive left and right eigenvectors;

\item[(iii)] the eigenvectors are unique up to constant multiples;

\item[(iv)] let $\mathbf{l}=(l_{j})$ and $\mathbf{r}=(r_{j})$ be the left and right eigenvectors for $\lambda$, then $\mathbf{l}\cdot\mathbf{r}=\underset{j=1}{\overset{\infty}{\sum}}l_{j}r_{j}<\infty$ if and only if $T$ is positive recurrent;

\item[(v)] if $T$ is positive recurrent, $\underset{n\rightarrow\infty}{\lim} \frac{T^{n}}{\lambda^{n}}=\mathbf{r}\mathbf{l}$, normalized so that $\mathbf{l}\cdot\mathbf{r}=1$.

\end{thm}

In the following, the Perron value of $T$ and its left and right eigenvectors can be obtained by finite approximation; see Theorem 7.1.4 in \cite{1-1}.
\begin{thm}
\label{theorem:5.2}
Let $T$ be countable, non-negative, irreducible, aperiodic and recurrent. Assume $\lambda$ is the Perron value of $T$. Let $\mathbf{l}$ and $\mathbf{r}$ be the left and right eigenvectors for $\lambda$, normalized so that $l_{i}=r_{i}=1$ for some index $i$.

\item[(i)] $\lambda=\left\{\lambda(A) \hspace{0.1cm}\mid\hspace{0.1cm}  A \text{ is a finite, irreducible and aperiodic submatrix of }T, \lambda(A) \text{ is the}\right.$

           $\left. \text{Perron value of }A \right\}$
\item[(ii)] Let $\{A_{n}\}$ be an increasing
family of finite irreducible, aperiodic submatrices of $T$ that converges to $T$. Let $\mathbf{l}^{(n)},\mathbf{r}^{n}$ be the left and right eigenvectors of $A_{n}$ for $\lambda(A_{n})$,
 normalized so that $l^{(n)}_{i}=r^{(n)}_{i}=1$. Then, $\underset{n\rightarrow \infty}{\lim}l_{j}^{(n)}=l_{j}$ and $\underset{n\rightarrow \infty}{\lim}r_{j}^{(n)}=r_{j}$ for all index $j$.
\end{thm}

As the finite state shift space in Section 2, countable state shift space is now introduced. Let $\mathcal{I}$ be the set of countable symbols. The two-sided full space is defined by

\begin{equation}\label{eqn:5.1}
\Sigma(\mathcal{I})\equiv \mathcal{I}^{\mathbb{Z}^{1}}=\left\{\left(x_{n}\right)_{n=-\infty}^{\infty} \hspace{0.1cm}\mid\hspace{0.1cm} x_{n}\in \mathcal{I}\right\}.
\end{equation}

Given a countable $0$-$1$ matrix $T$ with index set $\mathcal{I}$, the subshift $\Sigma(T)$ is defined by

\begin{equation}\label{eqn:5.2}
\Sigma(T)=\left\{x=\left(x_{n}\right)_{n=-\infty}^{\infty}\in\Sigma(\mathcal{I}) \hspace{0.1cm}\mid\hspace{0.1cm} T_{x_{n},x_{n+1}}=1 \text{for all }n\in\mathbb{Z}^{1}  \right\}.
\end{equation}
$\Sigma(T)$ is called a countable symbol shift of finite type.

A cylinder in $\Sigma(T)$ is defined by
\begin{equation}\label{eqn:5.3-1}
C([\alpha;i_{1},i_{2},\cdots, i_{n}])=\{x\in\Sigma(T)  \hspace{0.1cm}\mid\hspace{0.1cm} x_{\alpha}=i_{1}, x_{\alpha+1}=i_{2},\cdots, x_{\alpha+n-1}=i_{n}\}
\end{equation}
where $i_{k}\in\mathcal{I}$, $1\leq k\leq n$ and $\alpha\in\mathbb{Z}^{1}$. $[i_{1},i_{2},\cdots, i_{n}]$ is called an admissible pattern if $C([\alpha;i_{1},i_{2},\cdots, i_{n}])\neq\emptyset$ for some $\alpha\in\mathbb{Z}^{1}$. Define the set of all admissible patterns of length $n$ by

\begin{equation}\label{eqn:5.3-2}
\mathcal{B}_{n}(\Sigma(T))=\left\{[i_{1},i_{2},\cdots, i_{n}] \hspace{0.1cm}\mid\hspace{0.1cm} C([\alpha;i_{1},i_{2},\cdots, i_{n}])\neq\emptyset \text{ for some }\alpha\in\mathbb{Z}^{1}\right\}.
\end{equation}
For any $i,j\in\mathcal{I}$, denote by $\mathcal{B}_{n;i,j}(\Sigma(T))$ the admissible patterns in $\mathcal{B}_{n}(\Sigma(T))$ with initial state $i$ and terminal state $j$:

\begin{equation}\label{eqn:5.3-3}
\mathcal{B}_{n;i,j}(\Sigma(T))=\left\{[i_{1},i_{2},\cdots, i_{n}]\in \mathcal{B}_{n}(\Sigma(T)) \hspace{0.1cm}\mid\hspace{0.1cm} i_{1}=i \text{ and } i_{n}=j \right\}.
\end{equation}
Since $T$ is countably infinite, in general, $|\mathcal{B}_{n}(\Sigma(T))|=\infty$. However, $|\mathcal{B}_{n;i,j}(\Sigma(T))|$ is finite when the Perron value of $T$ is finite.

As the Parry measure for irreducible shifts of finite type, a Markov measure for $\Sigma(T)$, which is the unique measure with maximal entropy, is introduced; see the work of \cite{1-1}. If $T$
 is irreducible and positive recurrent, then the Markov measure is defined by the pair $(p,P)$ with
 \begin{equation}\label{eqn:5.3-4}
 \begin{array}{ccc}
 p_{i}=l_{i}r_{i} & \text{and} & P_{i,j}= T_{i,j}\frac{r_{i}}{\lambda r_{j}},
 \end{array}
 \end{equation}
where $\lambda$ is the Perron value of $T$ and $\mathbf{l}$ and $\mathbf{r}$ are the left and right eigenvectors of $\lambda$, normalized so that $\mathbf{l}\cdot \mathbf{r}=1$. Notably, $(p,P)$ is the unique measure with maximal entropy.

Now, if $T$ is irreducible, aperiodic and positive recurrent, then the natural measure of $\Sigma(T)$ can be defined as follows. For $i_{q}\in \mathcal{I}$, $1\leq q\leq n$, let

\begin{equation}\label{eqn:5.4}
\begin{array}{rl}
 & C_{k,l;i,j}([i_{1},i_{2},\cdots, i_{n}]) \\
 & \\
  = & \left\{ x=(x_{-k+1},\cdots,x_{-1},x_{0},x_{1},\cdots,x_{n+l} )\in  \mathcal{B} _{n+k+l}(\Sigma(T))\hspace{0.1cm}\mid\hspace{0.1cm}x_{-k+1}=i, x_{n+l}=j, x_{q}=i_{q}, 1\leq q\leq n \right\}.
\end{array}
\end{equation}

Then, define

\begin{equation}\label{eqn:5.5}
\begin{array}{rl}
 & \mu_{i,j}([i_{1},i_{2},\cdots, i_{n}]) \\
 & \\
=  & \underset{k,l\rightarrow\infty}{\lim} \left|C_{k,l;i,j}([i_{1},i_{2},\cdots, i_{n}]) \right|/ \left|\mathcal{B}_{n+k+l;i,j}(\Sigma(T)) \right|.
\end{array}
\end{equation}

Therefore, the following theorem can be immediately proven.

\begin{thm}
\label{theorem:5.3}
Let $T$ be irreducible, aperiodic and positive recurrent. Then, the natural measure exists and equals the Markov measure $(p,P)$.
\end{thm}

\begin{proof}
It can be easily verified that

\begin{equation*}
 \left|C_{k,l;i,j}([i_{1},i_{2},\cdots, i_{n}])\right|=\left(T^{k}\right)_{i,i_{1}} T_{i_{1},i_{2}}\cdots T_{i_{n-1},i_{n}} \left(T^{l}\right)_{i_{n},j}
\end{equation*}
and
\begin{equation*}
\left|\mathcal{B}_{n+k+l;i,j}(\Sigma(T)) \right|= \left(T^{n+k+l-1}\right)_{i,j}.
\end{equation*}
From the generalized Perron-Frobenious Theorem,

\begin{equation}\label{eqn:5.6}
\mu_{i,j}([i_{1},i_{2},\cdots, i_{n}])=\frac{l_{i_{1}}r_{i_{n}}}{\lambda^{n-1}}T_{i_{1},i_{2}}T_{i_{2},i_{3}}\cdots T_{i_{n-1},i_{n}},
\end{equation}
which is independent of $i$ and $j$. The natural measure of $\Sigma(T)$ is defined by

\begin{equation}\label{eqn:5.7}
\mu([i_{1},i_{2},\cdots, i_{n}])\equiv\mu_{i,j}([i_{1},i_{2},\cdots, i_{n}])
\end{equation}
for some $i,j\in\mathcal{I}$.

From (\ref{eqn:5.3-4}) and (\ref{eqn:5.6}), the result follows immediately.
\end{proof}
In Kitchens \cite{1-1}, Proposition 7.2.13 states that $\Sigma_{T}$ has measure with maximal entropy if and only if $T$ is positive recurrent. Therefore, if $T$ is null recurrent or transient, then the previous procedure for
finding natural measure will encounter some intrinsic problems. Indeed, the situation that pertains when $T$ is null-recurrent or transient is discussed briefly below.

The random walk on integers is investigated first.

\begin{ex}
\label{example:5.4}

The random walk on integers is defined as follows \cite{1-1}.

\begin{equation*}
\includegraphics[scale=0.4]{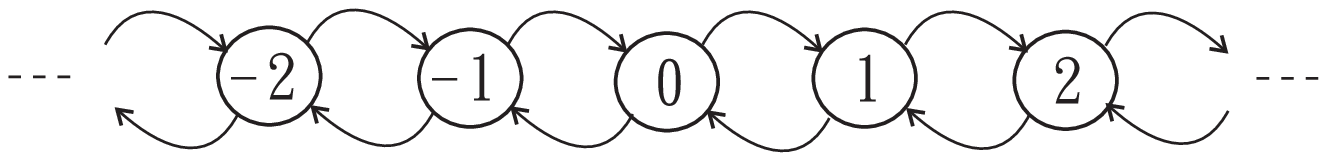}
\end{equation*}
\begin{equation*}
\text{Figure 5.1.}
\end{equation*}
The set of symbols is $\mathbb{Z}^{1}$ and the associated transition matrix $T$ is given by
\end{ex}

\begin{equation}\label{eqn:5.13}
\begin{array}{c}
\psfrag{a}{$T=$}
\psfrag{b}{{\footnotesize$\mathbb{Z}^{1}\times \mathbb{Z}^{1}$}}
\includegraphics[scale=0.5]{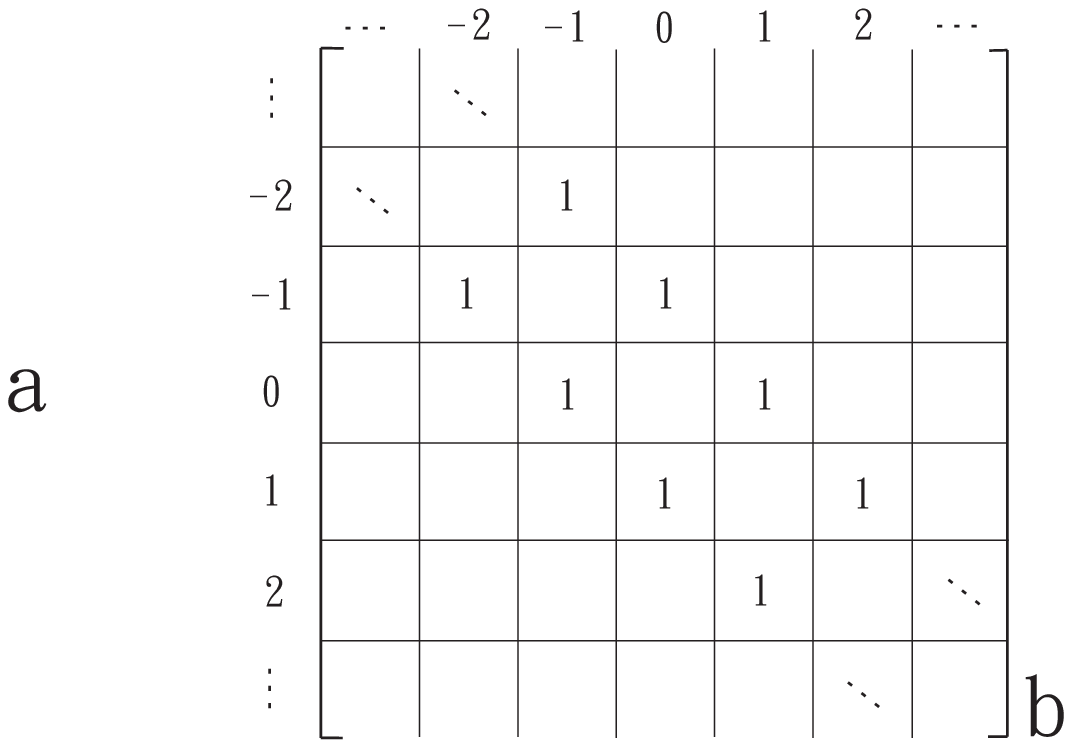}
\end{array}
\end{equation}

Clearly, the Perron value $\lambda=2$, and the left eigenvector $\mathbf{l}$ and right eigenvector $\mathbf{r}$ are

\begin{equation}\label{eqn:5.14}
\mathbf{l}=\mathbf{r}^{t}=(\cdots,1,1,1,\cdots),
\end{equation}
all entries are $1$.
Therefore, the transition matrix $T$ is null recurrent. It is easy to verify

\begin{equation}\label{eqn:5.15}
\left| \mathcal{B}_{n;i,j}\right|=\left|\left\{(\alpha_{1},\alpha_{2},\cdots,\alpha_{n-1})\hspace{0.1cm}\mid\hspace{0.1cm} \underset{k=1}{\overset{n-1}{\sum}}\alpha_{k}=j-i, \alpha_{k}\in\{0,1\}, 1\leq k\leq n-1 \right\}\right|.
\end{equation}
In particular, for $k,l\geq1$,

\begin{equation}\label{eqn:5.16}
\left| \mathcal{B}_{2k+2l+1;0,0}\right|=\left(T^{2k+2l}\right)_{0,0}=C^{2k+2l}_{k+l}
\end{equation}
and

\begin{equation}\label{eqn:5.17}
\left|C_{2k,2l;0,0}([0])\right|= \left(T^{2k}\right)_{0,0}\left(T^{2l}\right)_{0,0}=C^{2k}_{k}C^{2l}_{l}.
\end{equation}

By Stirling formula, it can be verified that
\begin{equation*}
 \frac{\left|C_{2k,2l;0,0}\right|([0])}{\left| \mathcal{B}_{2k+2l+1;0,0}\right|}=\frac{\left(T^{2k}\right)_{0,0}\left(T^{2l}\right)_{0,0}}{\left(T^{2k+2l}\right)_{0,0}}\sim \sqrt{\frac{k+l}{kl\pi }} \hspace{0.5cm} \text{as }k,l\rightarrow\infty,
\end{equation*}
which implies

\begin{equation}\label{eqn:5.18-0}
\underset{k,l\rightarrow\infty}{\lim} \frac{\left|C_{2k,2l;0,0}([0])\right|}{\left| \mathcal{B}_{2k+2l+1;0,0}\right|}=0.
\end{equation}
On the other hand, if $m$ is odd, then $\left| \mathcal{B}_{m+1;0,0}\right|=0$; if one of $m_{1}$ and $m_{2}$ is odd, then  $\left|C_{m_{1};m_{2},0,0}([0])\right|=0$.

Similarly, for any $i\in\mathbb{Z}^{1}$,
\begin{equation}\label{eqn:5.18-1}
\left\{
\begin{array}{ll}
\underset{k,l\rightarrow\infty}{\lim} \frac{\left|C_{2k,2l;0,0}([i])\right|}{\left| \mathcal{B}_{2k+2l+1;0,0}\right|}=0 & \text{if }i \text{ is even}, \\
& \\
\underset{k,l\rightarrow\infty}{\lim} \frac{\left|C_{2k+1,2l+1;0,0}([i])\right|}{\left| \mathcal{B}_{2k+2l+3;0,0}\right|}=0 & \text{if }i \text{ is odd}.
\end{array}
\right.
\end{equation}
In viewing (\ref{eqn:5.18-0}) and (\ref{eqn:5.18-1}), (\ref{eqn:5.5}) does not produce any measure.

For the case when $T$ is transient, recall the following shift space, which was discussed by Kitchens \cite{1-1}.

\begin{equation*}
\includegraphics[scale=0.4]{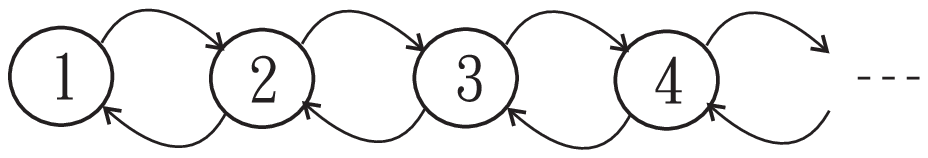}
\end{equation*}
\begin{equation*}
\text{Figure 5.2.}
\end{equation*}
The Perron value $\lambda=2$ and the left eigenvector $\mathbf{l}$ and right eigenvector $\mathbf{r}$ are given by
\begin{equation*}
\mathbf{l}=(1,2,3,\cdots)=\mathbf{r}^{t}.
\end{equation*}
It can also be shown that (\ref{eqn:5.5}) does not produce any measure.

Now, the natural measure of the countable state sofic shift is considered. First, the countable state labeled graph is introduced as follows.

Let $G_{\infty}$ be a graph with a countable set $\mathcal{V}(G_{\infty})$ of vertices and a countable set $\mathcal{E}(G_{\infty})$ of edges. For any $i,j\in \mathcal{V}(G_{\infty})$, denote by $T_{i,j}$ the number of edges in $G_{\infty}$ with initial state $i$ and terminal state $j$. The adjacency matrix of $G_{\infty}$ is defined by $T=T(G_{\infty})=[T_{i,j}]$.

Given a graph $G_{\infty}$, a labeling function $\mathcal{L}: \mathcal{E}(G_{\infty})\rightarrow \mathcal{S}$ assigns to each vertex $\xi\in \mathcal{V}(G_{\infty})$ a label $\mathcal{L}(\xi)\in\mathcal{S}$. A labeled graph is a pair $\mathcal{G}_{\infty}=(G_{\infty},\mathcal{L})$. A labeled graph $\mathcal{G}_{\infty}=(G_{\infty},\mathcal{L})$ is right-resolving if the edge with initial state $i$ carries different labels for every vertex $i\in \mathcal{V}(G)$. Let $\mathcal{T}_{i,j}$ be the formal sum of the alphabets from vertex $i$ to vertex $j$. Then, the symbolic adjacency matrix is defined by $\mathcal{T}=\mathcal{T}(\mathcal{G}_{\infty})=[\mathcal{T}_{i,j}]$. Denote by $\mathbb{T}$ the adjacency matrix of $\mathcal{T}$. For each $s\in\mathcal{S}$, define $\mathcal{T}(s)=[\mathcal{T}_{i,j}(s)]$ by

\begin{equation*}
 \left\{
 \begin{array}{ll}
 \mathcal{T}_{i,j}(s)=s & \text{ if } s\in  \mathcal{T}_{i,j} \\
 & \\
  \mathcal{T}_{i,j}(s)=\emptyset & \text{ otherwise.}
\end{array}
 \right.
\end{equation*}
Moreover, denote by $\mathbb{T}(s)$ the adjacency matrix of $\mathcal{T}(s)$.

As in Section 4, the edge shift $X_{G_{\infty}}$ of $G_{\infty}$ is defined by

\begin{equation}\label{eqn:5.8}
X_{G_{\infty}}=\left\{\xi=(\xi_{i})_{i=-\infty}^{\infty}\in\mathcal{E}^{\mathbb{Z}^{1}}(G_{\infty})\hspace{0.1cm}\mid \hspace{0.1cm}t(\xi_{i})=i(\xi_{i+1}) \text{ for all }i\in\mathbb{Z}^{1}\right\},
\end{equation}
and the sofic shift $X_{\mathcal{G}_{\infty}}$ is defined by

\begin{equation}\label{eqn:5.9}
\begin{array}{rl}
X_{\mathcal{G}_{\infty}}= & \left\{x\in\mathcal{S}^{\mathbb{Z}^{1}} \hspace{0.1cm}\mid\hspace{0.1cm} x=\mathcal{L}_{\infty}(\xi)\equiv\left(\mathcal{L}(\xi_{n})\right)_{n=-\infty}^{\infty} \text{ where }
\xi=\left(\xi_{n}\right)_{n=-\infty}^{\infty}\in X_{G_{\infty}} \right\} \\
& \\
= & \mathcal{L}_{\infty}(X_{G_{\infty}}).
\end{array}
\end{equation}

For any vertices $i,j\in\mathcal{V}(G)$, the cylinder set with initial state $i$ and terminal state $j$ is defined by
\begin{equation}\label{eqn:5.10}
\begin{array}{rl}
& C_{k,l;i,j}([s_{i_{1}},s_{i_{2}},\cdots,s_{i_{n}} ]) \\
 & \\
 = & \left\{ x=(x_{-k+1},\cdots,x_{-1},x_{0},x_{1},\cdots,x_{n+l} )\in  \mathcal{B}_{n+k+l}(X_{\mathcal{G}_{\infty}})\hspace{0.1cm}\mid\hspace{0.1cm}i(x_{-k+1})=i, t(x_{n+l})=j,\right.\\
 & \\
 & \left. x_{q}=s_{i_{q}}, 1\leq q\leq n \right\}
\end{array}
\end{equation}
where $\mathcal{B}_{n}(X_{\mathcal{G}_{\infty}})$ is the set of all admissible patterns with length $n$ in $(X_{\mathcal{G}_{\infty}})$. The set of the admissible patterns in $\mathcal{B}_{n}(X_{\mathcal{G}_{\infty}})$ with initial state $i$ and terminal state $j$ is denoted by $\mathcal{B}_{n;i,j}(X_{\mathcal{G}_{\infty}})$

Given a right-resolving countable state sofic shift $\mathcal{G}_{\infty}=(G_{\infty},\mathcal{L})$, if $T$ is irreducible, aperiodic and positive recurrent, then the natural measure of initial state $i$ and terminal state $j$ is defined as follows.
For any admissible pattern $\left[s_{i_{1}},s_{i_{2}},\cdots,s_{i_{n}}\right]$,

\begin{equation}\label{eqn:5.11}
\mu_{i,j}\left(\left[s_{i_{1}},s_{i_{2}},\cdots,s_{i_{n}}\right]\right)\equiv\underset{k,l\rightarrow\infty}{\lim} \left|C_{k,l;i,j}([s_{i_{1}},s_{i_{2}},\cdots,s_{i_{n}}]) \right|/ \left|\mathcal{B}_{n+k+l;i,j}(X_{\mathcal{G}_{\infty}}) \right|.
\end{equation}

\begin{thm}
\label{theorem:5.4}
Suppose $\mathcal{G}_{\infty}=(G_{\infty},\mathcal{L})$ is right-resolving. Let $T$ be irreducible, aperiodic and positive recurrent. Let $\lambda$ be the Perron value of $T$ and $\mathbf{l}=(l_{j})$ and $\mathbf{r}=(r_{j})$ be the left and right eigenvectors for $\lambda$, normalized with $\mathbf{l}\cdot\mathbf{r}=1$. The natural measure
\begin{equation}\label{eqn:5.12}
\begin{array}{rl}
\mu\left(\left[s_{i_{1}},s_{i_{2}},\cdots,s_{i_{n}}\right]\right)\equiv & \mu_{i,j}\left(\left[s_{i_{1}},s_{i_{2}},\cdots,s_{i_{n}}\right]\right)\\
& \\
=& \underset{q,q'}{\sum}l_{q}\left[\mathbb{T}(s_{i_{1}})\mathbb{T}(s_{i_{2}})\cdots\mathbb{T}(s_{i_{n}})\right]_{q,q'}r_{q'}/\lambda^{n}
\end{array}
\end{equation}
is independent of $i$ and $j$.
\end{thm}

\begin{proof}
Since $\mathcal{G}_{\infty}$ is right-resolving, it can be verified that

\begin{equation*}
\left|C_{k,l;i,j}([s_{i_{1}},s_{i_{2}},\cdots,s_{i_{n}} ])\right|=
\underset{q,q'}{\sum}\left(\mathbb{T}^{k}\right)_{i,q}\left[\mathbb{T}(s_{i_{1}})\mathbb{T}(s_{i_{2}})\cdots\mathbb{T}(s_{i_{n}})\right]_{q,q'}\left(\mathbb{T}^{l}\right)_{q',j}
\end{equation*}
and

\begin{equation*}
\left|\mathcal{B}_{n+k+l;i,j}(X_{\mathcal{G}_{\infty}})\right|=\left(\mathbb{T}^{n+k+l}\right)_{i,j}.
\end{equation*}
From (v) of the generalized Perron-Frobenius Theorem,

\begin{equation*}
\left|C_{k,l;i,j}([s_{i_{1}},s_{i_{2}},\cdots,s_{i_{n}} ])\right|\sim
\lambda^{k+l}r_{i}l_{j}\left(\underset{q,q'}{\sum}l_{q}\left[\mathbb{T}(s_{i_{1}})\mathbb{T}(s_{i_{2}})\cdots\mathbb{T}(s_{i_{n}})\right]_{q,q'}r_{q'}\right)
\end{equation*}
and
\begin{equation*}
\left|\mathcal{B}_{n+k+l;i,j}(X_{\mathcal{G}_{\infty}})\right|\sim \lambda^{n+k+l}r_{i}l_{j}.
\end{equation*}
Therefore, (\ref{eqn:5.12}) follows immediately.
\end{proof}
%
%
%
%
%

\section{General shift space}
\setcounter{equation}{0}

In this section, Krieger cover is utilized to study the natural measure of the general shift space by transforming it into a countably infinite state sofic shift.

In \cite{1}, the extension of irreducible sofic shifts to general shift spaces has been discussed. Indeed, three extensions have been considered:

\begin{itemize}
\item[(1)] $X$ is a sofic shift.

\item[(2)] $X$ has countably many futures of left-infinite sequences.

\item[(3)] $X$ has an intrinsically synchronizing word.

\item[(4)] $X$ is a coded system.
\end{itemize}
For irreducible shifts (1)$\Rightarrow$ (2)$\Rightarrow$ (3)$\Rightarrow$ (4) and there are examples to show that none of these implications can be reversed. Based on the study of countably infinite state sofic shift of the previous section, (2) is investigated as follows.

Denote by $\Sigma$ the subshift space with a finite set $\mathcal{A}$ of symbols. In general, $\Sigma$ is not a shift of finite type or a sofic shift. $\Sigma$ is known to be able to be specified by describing its forbidden set:

\begin{equation}\label{eqn:6.2}
\Sigma=\Sigma_{\mathcal{F}}= \left\{
 x=(x_{n})_{n=-\infty}^{\infty}\in \mathcal{A}^{\mathbb{Z}^{1}} \hspace{0.1cm}\mid \hspace{0.1cm} (x_{n})_{n=k}^{l} \text{ is not in } \mathcal{F} \text { for all integers }  l>k
 \right\},
\end{equation}
where

\begin{equation}\label{eqn:6.1}
\mathcal{F}\subset \underset{n=1}{\overset{\infty}{\bigcup}} \mathcal{B}_{n}(\Sigma).
\end{equation}

If $\Sigma=\Sigma_{\mathcal{F}'}$ with finite $\mathcal{F}'$, then $\Sigma$ is a shift of finite type.
Therefore, $\mathcal{F}$ is assumed to be countably infinite in this section. For simplicity, a shift space $\Sigma$ is assumed to be irreducible here.

Define the set $\Sigma^{-}$ of left-sequence by

\begin{equation}\label{eqn:6.3}
\Sigma^{-}= \left\{
 x=(x_{n})_{n=-\infty}^{0}   \hspace{0.1cm}\mid \hspace{0.1cm}  (x_{n})_{n=-\infty}^{\infty}\in \Sigma
 \right\}
\end{equation}
and the set $\Sigma^{+}$ of right-sequence by

\begin{equation}\label{eqn:6.4}
\Sigma^{+}= \left\{
 x=(x_{n})_{n=1}^{\infty}   \hspace{0.1cm}\mid \hspace{0.1cm}  (x_{n})_{n=-\infty}^{\infty}\in \Sigma
 \right\}.
\end{equation}

The Krieger (future) cover is introduced as follows. For each $\xi \in \Sigma^{-}$,
the future cover of $\xi$ is denoted as

\begin{equation}\label{eqn:6.5}
F(\xi)=\left\{
\eta\in \Sigma^{+}  \hspace{0.1cm}\mid \hspace{0.1cm} \xi\eta \in \Sigma
\right\}.
\end{equation}
An equivalent relation $\sim$ on $\Sigma^{-}$ can be defined as follows. For $\xi_{1}$ and $\xi_{2}$ in $\Sigma^{-}$, $\xi_{1}$ is equivalent to $\xi_{2}$, i.e., $\xi_{1}\sim \xi_{2}$ if

\begin{equation*}
F(\xi_{1})=F(\xi_{2}).
\end{equation*}
Therefore, $\mathfrak{I}=\Sigma^{-}/\sim=\left\{[\xi]  \hspace{0.1cm}\mid \hspace{0.1cm}   \xi\in \Sigma^{-} \right\}$ is decomposed into equivalent classes.
$[\xi]$ is now treated as a state. For simplicity, use $\xi$ instead of $[\xi]$ as long doing so does not cause confusion.
Notably, $\mathfrak{I}$ may be uncountable.

With these states (vertices), the labeled graph $\mathcal{G}_{\infty}=(G_{\infty},\mathcal{L})$ is defined by for any $s\in \mathcal{A}$,

\begin{equation*}
\xi_{1} \overset{s}{\rightarrow} \xi_{2}
\end{equation*}
whenever $\xi_{2}=\xi_{1}s\in\mathfrak{I}$, where $\mathfrak{I}$ is countable.

 When the Krieger cove of $\Sigma$ induces a countable state sofic shift $\mathcal{G}_{\infty}=(G_{\infty},\mathcal{L})$
and the associated adjacency matrix $\mathbb{T}$ is irreducible, aperiodic and positive recurrent, by Theorem 5.5, the natural measure of $\mathcal{G}_{\infty}=(G_{\infty},\mathcal{L})$ exists and can be represented by (\ref{eqn:5.12}). Now, we can define a measure $\mu$ on $\Sigma$ by (\ref{eqn:5.12}). Furthermore, we still call $\mu$ the natural measure of $\Sigma$.

%
%
%
%

\begin{thm}
\label{theorem:6.0-1}
Let $\Sigma$ be a shift space. Suppose $\mathcal{G}_{\infty}=(G_{\infty},\mathcal{L})$ is the labeled graph constructed from $\Sigma$ by Krieger cover. Assume $\mathcal{G}_{\infty}$ is right-resolving and its associated adjacency matrix $\mathbb{T}$ is countable, irreducible, aperiodic and positive recurrent. Then, (\ref{eqn:5.12}) define a measure on $\Sigma$.
\end{thm}

Notably, the measure $\mu$ on $\Sigma$ is induced from $\mathcal{G}_{\infty}=(G_{\infty},\mathcal{L})$. For this $\mu$, we do not compute the number of finite cylinder $C_{k,l}([i_{1},i_{2},\cdots,i_{n}])$ in (\ref{eqn:1.5}) directly. Instead, we compute the number of cylinder $C_{k,l;i,j}([s_{i_{1}},s_{i_{2}},\cdots,s_{i_{n}} ])$ of $\mathcal{G}_{\infty}=(G_{\infty},\mathcal{L})$ which is given by (\ref{eqn:5.4}) .

In the following, under certain conditions on $\Sigma$, we show that the natural measure $\mu$ of $\Sigma$ defined by (\ref{eqn:5.12}) attains the maximal entropy, i.e., the measure theoretical entropy of $\mu$ on $\Sigma$ equals the topological entropy $h(\Sigma)$ of $\Sigma$.
%

Let $\mathcal{T}=[t_{i,j}]_{i,j\in\mathbb{N}}$ be the symbolic adjacency matrix of $\mathcal{G}_{\infty}=(G_{\infty},\mathcal{L})$.
For any $n\geq 1$, define the $n$-th order symbolic adjacency matrix $\mathcal{T}_{n}= [t_{n;i,j}]_{n\times n}$ by $t_{n;i,j}=t_{i,j}$ for $1\leq i,j\leq n$. For any $n\geq 1$, let $\mathcal{G}_{n}$ be the sofic shift that is induced by $\mathcal{T}_{n}$. Clearly, if $\mathcal{G}_{\infty}$ is right-resolving, then  $\mathcal{G}_{n}$ is right-resolving. Denote by $\mathbb{T}_{n}$ the transition matrix of $\mathcal{T}_{n}$. Notably, $\mathcal{G}_{k}\subseteq \mathcal{G}_{l}$ for $1\leq k < l$, and $\mathcal{G}_{\infty}=\underset{n=1}{\overset{\infty}{\bigcup}} \mathcal{G}_{n}$.

 Suppose $\Sigma$ is irreducible. Then, for any $U_{n}=(u_{1},u_{2},\cdots,u_{n})\in\Sigma_{n}$, there exists a global pattern $W=(w_{k})_{k=-\infty}^{\infty}\in\Sigma$ with $w_{j}=u_{j}$ for $1\leq j\leq n$. Let $\xi_{0}=(w_{n})_{n=-\infty}^{0}\in \Sigma^{-}$.
From the definition of $\mathcal{G}_{\infty}=(G_{\infty},\mathcal{L})$, there exist $[\xi_{j}]$, $1\leq j\leq n$, such that

\begin{equation*}
[\xi_{0}] \overset{ u_{1}}{\rightarrow}[\xi_{1}] \overset{u_{2}}{\rightarrow}\cdots\overset{u_{n}}{\rightarrow} [\xi_{n}].
\end{equation*}
Therefore,

\begin{equation*}
\mathcal{B}_{n}(X_{\mathcal{G}_{\infty}})=\Sigma_{n}
\end{equation*}
for all $n\geq 1$. Moreover, let $N^{\ast}_{n}$ be the smallest number such that

\begin{equation}\label{eqn:6.6-6}
\mathcal{B}_{n}(\mathcal{G}_{\infty})=\underset{1\leq i,j\leq N^{\ast}_{n}}{\bigcup}\mathcal{B}_{n;i,j}(\mathcal{G}_{\infty}).
\end{equation}
%

Now, the following result can be obtained.

\begin{thm}
\label{theorem:6.2}
Suppose $\Sigma$ is irreducible. Let $\mathcal{G}_{\infty}=(G_{\infty},\mathcal{L})$ be the labeled graph constructed from $\Sigma$ by Krieger cover. Assume $\mathcal{G}_{\infty}$ is right-resolving and its associated adjacency matrix $\mathbb{T}$ is countable, irreducible, aperiodic and positive recurrent. Let $\lambda$ be the Perron value of $\mathbb{T}$. If
\item[(i)]

\begin{equation*}
\underset{n\rightarrow\infty}{\limsup}\frac{ \log N_{n}^{\ast}}{n}=0;
\end{equation*}
\item[(ii)] there exist $i^{\ast}$ and $j^{\ast}\in\mathbb{N}$ such that

\begin{equation*}
\left(\mathbb{T}^{n}\right)_{i^{\ast},j^{\ast}}\geq \left(\mathbb{T}^{n}\right)_{i,j}
\end{equation*}
for all $n\geq 1$ and $(i,j)\neq (i^{\ast},j^{\ast})$, then the topological entropy $h(\Sigma)=\log \lambda$.
\end{thm}
\begin{proof}

Let $\lambda_{n}$ be the Perron value of $\mathbb{T}_{n}$.
 it is easy to see that $\left| \mathcal{B}_{n}(\mathcal{G}_{\infty})  \right|\geq \left| \mathcal{B}_{n}(\mathcal{G}_{m})  \right|$ for all $n\geq1$. Given $m\geq 1$, since $\mathcal{G}_{m}$ is right-resolving,

\begin{equation*}
\underset{n\rightarrow \infty}{\limsup} \frac{1}{n}\log \left| \mathcal{B}_{n}(\mathcal{G}_{\infty})  \right|\geq \underset{m\geq 1}{\sup} h(\mathcal{G}_{m})=\underset{m\geq 1}{\sup} \log \lambda_{m}.
\end{equation*}
From Theorem 5.2,

\begin{equation*}
h(\Sigma)=\underset{n\rightarrow \infty}{\limsup} \frac{1}{n}\log \left| \mathcal{B}_{n}(\mathcal{G}_{\infty})  \right|\geq\underset{m\geq 1}{\sup} \log \lambda_{m}=\log \lambda.
\end{equation*}

Next, from (\ref{eqn:6.6-6}), (i) and (ii),
\begin{equation*}
\begin{array}{rl}
h(\Sigma)=\underset{n\rightarrow \infty}{\limsup} \frac{1}{n}\log \left| \mathcal{B}_{n}(\mathcal{G}_{\infty})  \right|\leq  &  \underset{n\rightarrow \infty}{\limsup} \frac{1}{n}\log \left( \underset{1\leq i,j\leq N_{n}^{\ast}}{\sum} \left| \mathcal{B}_{n;i,j}(\mathcal{G}_{\infty})  \right|\right) \\
& \\
\leq&   \underset{n\rightarrow \infty}{\limsup} \frac{1}{n}\log \left(N_{n}^{\ast}\right)^{2}\left( \mathbb{T}^{n} \right)_{i^{\ast},j^{\ast}} \\
& \\
=& \log\lambda.
\end{array}
\end{equation*}
Therefore, the result holds.

%
%

%
%
%
%

\end{proof}

The following result provides sufficient condition for the uniform distribution property (\ref{eqn:4.8-2}) of the natural measure.

\begin{thm}
\label{theorem:6.3}
Suppose $\mathcal{G}_{\infty}=(G_{\infty},\mathcal{L})$ is the labeled graph constructed from $\Sigma$ by Krieger cover. Assume $\mathcal{G}_{\infty}$ is right-resolving and its associated adjacency matrix $\mathbb{T}$ is countable, irreducible, aperiodic and positive recurrent. Let $\lambda$ be the Perron value of $\mathbb{T}$ and $\mathbf{l}=(l_{j})$ and $\mathbf{r}=(r_{j})$ be the left and right eigenvectors for $\lambda$, normalized with $\mathbf{l}\cdot\mathbf{r}=1$. If
\item[(i)] there exist $m_{1},m_{2}>0$ such that

\begin{equation}\label{eqn:6.6-51}
m_{1}\leq r_{j}\leq m_{2}
\end{equation}
for all $j\geq 1$,

\item[(ii)] there exists a finite set $\mathcal{N}\subset\mathbb{N}$ such that for any admissible pattern $\left[  s_{i_{1}}, s_{i_{2}},\cdots, s_{i_{n}}\right]$, there exists $q\in\mathcal{N}$ such that

\begin{equation}\label{eqn:6.6-52}
\left[\mathbb{T}(s_{i_{1}})\mathbb{T}(s_{i_{2}})\cdots \mathbb{T}(s_{i_{n}})\right]_{q,q'}=1,
\end{equation}
then the natural measure that is defined by (\ref{eqn:5.12}) satisfies the uniform distribution property (\ref{eqn:4.8-2}), i.e.,
there exist constants $\beta\geq\alpha > 0$ such that

\begin{equation*}
\frac{\alpha}{\lambda^{n}}\leq \mu\left([s_{i_{1}},s_{i_{2}},\cdots,s_{i_{n}}]\right)\leq\frac{\beta}{\lambda^{n}}
\end{equation*}
for any admissible pattern $[s_{i_{1}},s_{i_{2}},\cdots,s_{i_{n}}]$.

\end{thm}

\begin{proof}

Since $\mathcal{G}_{\infty}$ is right-resolving, for any $q\geq 1$ and any admissible pattern $\left[ s_{i_{1}}, s_{i_{2}},\cdots, s_{i_{n}} \right]$ in $\Sigma$, there is at most one $q'\geq 1$ such that
(\ref{eqn:6.6-52}) holds. Clearly, from (\ref{eqn:6.6-51}), $\mathbf{l}\cdot\mathbf{r}=1$ implies

\begin{equation*}
0<\underset{q\geq 1}{\sum}l_{q}<\infty.
\end{equation*}
Let $\beta= m_{2}\underset{q\geq 1}{\sum}l_{q}$.
Then, by (i), for any admissible pattern $\left[ \left( s_{i_{1}}, s_{i_{2}},\cdots, s_{i_{n}} \right)\right]$,
\begin{equation*}
\begin{array}{rl}
\mu\left(\left[  s_{i_{1}}, s_{i_{2}},\cdots, s_{i_{n}}\right]\right)= &  \underset{q,q'}{\sum}l_{q}\left[\mathbb{T}(s_{i_{1}})\mathbb{T}(s_{i_{2}})\cdots\mathbb{T}(s_{i_{n}})\right]_{q,q'}r_{q'}/\lambda^{n} \\
& \\
\leq & m_{2}\underset{q,q'}{\sum}l_{q}\left[\mathbb{T}(s_{i_{1}})\mathbb{T}(s_{i_{2}})\cdots\mathbb{T}(s_{i_{n}})\right]_{q,q'}/\lambda^{n} \\
& \\
\leq &  m_{2}\underset{q}{\sum}l_{q}/\lambda^{n}=\beta/\lambda^{n}.
\end{array}
\end{equation*}

On the other hand, let
 \begin{equation*}
\alpha=m_{1}\left(\underset{q\in\mathcal{N}}{\min}\{l_{q}\}\right).
 \end{equation*}
 Then, by (i) and (ii), for any admissible pattern $\left[  s_{i_{1}}, s_{i_{2}},\cdots, s_{i_{n}}\right]$,
\begin{equation*}
\begin{array}{rl}
\mu\left(\left[  s_{i_{1}}, s_{i_{2}},\cdots, s_{i_{n}}\right]\right)= &  \underset{q,q'}{\sum}l_{q}\left[\mathbb{T}(s_{i_{1}})\mathbb{T}(s_{i_{2}})\cdots\mathbb{T}(s_{i_{n}})\right]_{q,q'}r_{q'}/\lambda^{n} \\
& \\
\geq & m_{1} \underset{q,q'}{\sum}l_{q}\left[\mathbb{T}(s_{i_{1}})\mathbb{T}(s_{i_{2}})\cdots\mathbb{T}(s_{i_{n}})\right]_{q,q'}/\lambda^{n} \\
& \\
\geq & m_{1} \left(\underset{q\in \mathcal{N}}{\min}\{l_{q}\}/\lambda^{n}\right)=\alpha/\lambda^{n}.
\end{array}
\end{equation*}

Therefore, the result follows.
\end{proof}

In the following theorem, we prove that when the natural measure $\mu$ of $\Sigma$ is uniformly distributed, then it is the only measure with ergodicity and maximal entropy.

\begin{thm}
\label{theorem:6.4}
Suppose $\Sigma$ is irreducible. Let $\mathcal{G}_{\infty}=(G_{\infty},\mathcal{L})$ be the labeled graph constructed from $\Sigma$ by Krieger cover. Assume $\mathcal{G}_{\infty}$ is right-resolving and its associated adjacency matrix $\mathbb{T}$ is countable, irreducible, aperiodic and positive recurrent.
Let $\lambda$ be the Perron value of $\mathbb{T}$.
If $h(\Sigma)=\log \lambda$ and the natural measure $\mu$ of $\Sigma$ satisfies (\ref{eqn:4.8-2}), then the natural measure is the only measure with ergodicity and maximal entropy.
\end{thm}

\begin{proof}
This proof is similar to that of Theorem 4.9. First, from (v) of Theorem 5.1, the ergodicity can be obtained. Since $h(\Sigma)=\log \lambda$ and $\mu$ satisfies (\ref{eqn:4.8-2}), (\ref{eqn:4.11-1}) yields that $\mu$ has maximal entropy $\log \lambda$. From the uniform distribution property, the uniqueness can be obtained as the proof of Theorem 8.10 in Walters \cite{2}.
%
%
%
%

\end{proof}
%
%
%
%
%
%
%
%
%
%

To illustrate Theorem 6.4, the following context free shift \cite{1} is investigated.

\begin{ex}
\label{example:6.2}
Consider the context free shift $\Sigma_{cf}$ whose forbidden set $\mathcal{F}$ is given by

\begin{equation}\label{eqn:6.7}
\mathcal{F}=\left\{
ab^{k}c^{l}a \hspace{0.1cm}\mid \hspace{0.1cm}  k\neq l \text{ for }k,l\geq 0
\right\}.
\end{equation}
The countable states are defined as follows.
\begin{equation*}
E_{j}=\left\{
\xi a  \hspace{0.1cm}\mid \hspace{0.1cm} \xi\in \Sigma^{-}
\right\}, \hspace{0.2cm}j\geq 0,
\end{equation*}

\begin{equation*}
F_{j}=\left\{
\xi ab^{m+j}c^{m}  \hspace{0.1cm}\mid \hspace{0.1cm} \xi\in \Sigma^{-}, m\geq 1
\right\}, \hspace{0.2cm}j\geq 0,
\end{equation*}

\begin{equation*}
P=\left\{
\xi ac^{m}b^{k}c^{l}  \hspace{0.1cm}\mid \hspace{0.1cm} \xi\in \Sigma^{-}, m,k\geq 1 \text{ and }l\geq 0
\right\}
\end{equation*}
and

\begin{equation*}
Q=\left\{
\xi ab^{i}c^{i+j}  \hspace{0.1cm}\mid \hspace{0.1cm} \xi\in \Sigma^{-}, i\geq 0 \text{ and }j\geq 1
\right\}.
\end{equation*}
Then, the labeled graph $\mathcal{G}_{\infty}=(G_{\infty},\mathcal{L})$ that represents (\ref{eqn:6.7}) is drawn as follows.

\begin{equation*}
\psfrag{a}{\scriptsize{$a$}}
\psfrag{b}{\scriptsize{$b$}}
\psfrag{c}{\scriptsize{$c$}}
\psfrag{p}{\small{$P$}}
\psfrag{q}{\small{$Q$}}
\psfrag{e}{\small{$F_{0}$}}
\psfrag{f}{\small{$F_{1}$}}
\psfrag{g}{\tiny{$F_{n-1}$}}
\psfrag{h}{\small{$F_{n}$}}
\psfrag{k}{\small{$E_{0}$}}
\psfrag{l}{\small{$E_{1}$}}
\psfrag{m}{\small{$E_{2}$}}
\psfrag{n}{\small{$E_{n}$}}
\psfrag{o}{\tiny{$E_{n+1}$}}
\psfrag{z}{\small{$\vdots$}}
\includegraphics[scale=0.7]{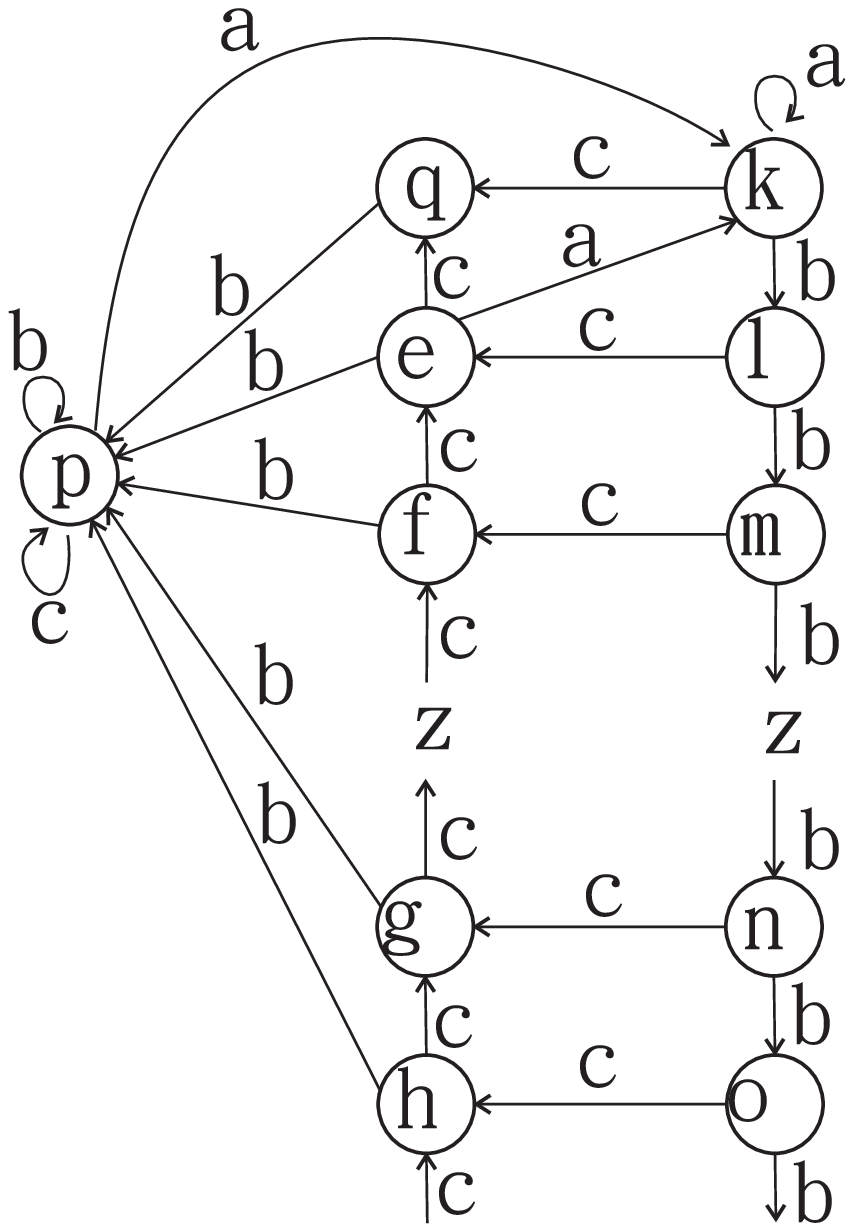}
\end{equation*}

\begin{equation*}
\text{Figure 6.1}
\end{equation*}

According to Fig. 6.1, the symbolic adjacency matrix $\mathcal{T}=[\mathcal{T}_{i,j}]_{i,j\in\mathbb{N}}$ of $\mathcal{G}_{\infty}=(G_{\infty},\mathcal{L})$ is given by

\begin{equation*}
\psfrag{a}{\scriptsize{$a$}}
\psfrag{b}{\scriptsize{$b$}}
\psfrag{c}{\scriptsize{$c$}}
\psfrag{d}{\tiny{$b+c$}}
\psfrag{f}{\small{$P$}}
\psfrag{g}{\small{$Q$}}
\psfrag{h}{\small{$E_{0}$}}
\psfrag{j}{\small{$E_{1}$}}
\psfrag{k}{\small{$F_{0}$}}
\psfrag{l}{\small{$E_{2}$}}
\psfrag{m}{\small{$F_{1}$}}
\psfrag{n}{\small{$E_{3}$}}
\psfrag{o}{\small{$F_{2}$}}
\psfrag{t}{\small{$\mathcal{T}=$}}
\psfrag{x}{\small{$\vdots$}}
\psfrag{y}{\small{$\cdots$}}
\psfrag{z}{\small{$\ddots$}}
\includegraphics[scale=1.1]{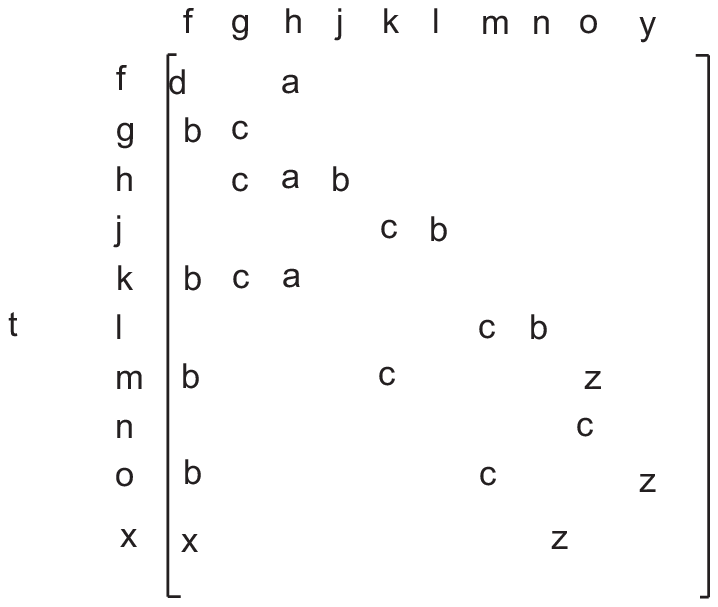}
\end{equation*}

$C^{*}$-algebra was used to show that the topological entropy $h(\Sigma_{cf})=\log \left(1+\sqrt{1+\sqrt{3}}\right)$ \cite{102}. Let $\mathbb{T}=[\mathbb{T}_{i,j}]$ be the adjacency matrix of $\mathcal{T}$.
That $\mathbb{T}$ is irreducible and aperiodic can be easily determined. Let $\lambda$ be the Perron value of $\mathbb{T}$. It remains to show that $\mathbb{T}$ is positive
recurrent, such that the left eigenvector $\mathbf{l}=(l_{j})$ and right eigenvector $\mathbf{r}=(r_{j})$ for $\lambda$ satisfy $\mathbf{l}\cdot\mathbf{r}<\infty$.

To show $\mathbb{T}$ is positive recurrent, for $n\geq 1$, define the finite submatrix $A_{n}=[a_{n;i,j}]_{(2n+3)\times (2n+3)}$ of $\mathbb{T}$ by

\begin{equation*}
a_{n;i,j}=\mathbb{T}_{i,j}
\end{equation*}
for $1\leq i,j\leq 2n+3$. Clearly, $A_{n}$ is irreducible and aperiodic.

Let $\lambda_{n}$ be the Perron value of $A_{n}$. Let $\mathbf{l}^{(n)}=(l^{(n)}_{1},l^{(n)}_{2},\cdots,l^{(n)}_{2n+3})$ and $\mathbf{r}^{(n)}=(r^{(n)}_{1},r^{(n)}_{2},\cdots,r^{(n)}_{2n+3})^{t}$ be the left and right eigenvectors for $\lambda_{n}$ with $l^{(n)}_{1}=r^{(n)}_{1}=1$. Notably, $\lambda> \lambda_{n}> 2$ for all $n\geq 2$.

The following five results can be obtained immediately.
\newline
\item[(i)] $\frac{l^{(n)}_{2k}}{l^{(n)}_{2k+2}}=\lambda_{n}>2$ for $n\geq2$ and $2\leq k \leq n$;
\newline
\item[(ii)] $\frac{l^{(n)}_{2k+1}}{l^{(n)}_{2k+3}}>\lambda_{n}>2$ for $n\geq2$ and $2\leq k \leq n$;
\newline
\item[(iii)] $r_{2}^{(n)}> \frac{1}{\lambda_{n}}$ and $r_{3}^{(n)}> \frac{1}{\lambda_{n}^{2}}$ for $n\geq2$,
\newline
\item[(iv)] $r^{(2n)}_{2k+3}>\frac{1}{\lambda_{n}}$ for $n\geq2$ and $2\leq k \leq 2n$;
\newline
\item[(v)] $r^{(2n)}_{2k+2}>\frac{1}{\lambda^{2}_{n}}$ for $n\geq2$ and $2\leq k \leq 2n$.
\newline

By Theorem 5.2, $\mathbf{l}\cdot\mathbf{r}<\infty$ with $l_{1}=r_{1}=1$ can be shown. Then, $\mathbb{T}$ is positive recurrent. Hence, the context free shift has the natural measure that is defined by (5.15).

Based on Fig. 6.1, it is easy to see that the admissible pattern $ab^{n}c^{n}a$ can be generated from $[\mathcal{T}_{i,j}]_{1\leq i,j\leq 2n+3}$. Indeed,

\begin{equation*}
P\overset{a}{\rightarrow}E_{0}\overset{b}{\rightarrow}E_{1}\overset{b}{\rightarrow}\cdots\overset{b}{\rightarrow} E_{n}\overset{c}{\rightarrow}F_{n-1}\overset{c}{\rightarrow}\cdots \overset{c}{\rightarrow} F_{0} \overset{a}{\rightarrow}E_{0}.
\end{equation*}
Then, that $N_{n}^{*}\leq n+1$ can be verified. Clearly, $\left(\mathbb{T}^{n}\right)_{1,1}\geq \left(\mathbb{T}^{n}\right)_{i,j}$ for all $n\geq 1$ and $(i,j)\neq(1,1)$. Therefore, from Theorem 6.2,
\begin{equation*}
h(\Sigma_{cf})=\log \left(1+\sqrt{1+\sqrt{3}}\right)=\log \lambda.
\end{equation*}

Next, the uniform distribution property is proven as follows. From (iii)$\sim$(v),
\begin{equation}\label{eqn:6.8}
\frac{1}{\lambda^{2}}\leq r_{j}\leq 1
\end{equation}
for all $j\geq1$. Based on Fig.6.1, it can be verified that for any admissible pattern $\left[ s_{i_{1}}, s_{i_{2}},\cdots, s_{i_{n}}\right]$, there exist $q\in \{1,3\}$ and $q\geq 1$ such that (\ref{eqn:6.6-52}) holds.
Hence, by Theorem 6.3, the natural measure satisfies the uniform distribution property.

Therefore, from Theorem 6.4, the natural measure is the only measure with ergodicity and maximal entropy.

\end{ex}

As in studying the countable state shifts, the positive recurrent condition is essential to ensure the existence of natural measure and the natural measure is represented in the form (\ref{eqn:5.12}). When the Krieger cover of a shift spaces is not positive recurrent, then the existence of the natural measure cannot be guaranteed and the natural measure cannot be represented in the form of (\ref{eqn:5.12}).

\bibliographystyle{amsplain}

\end{document}